\documentclass[12pt]{article}
\usepackage{amsmath}\usepackage{amssymb}\usepackage{amsfonts}\usepackage{amsthm}
\usepackage{color}
\usepackage{setspace}
\usepackage{bbm}
\usepackage{graphicx}
\usepackage[english]{babel}\usepackage[latin1]{inputenc}\usepackage{times}\usepackage[T1]{fontenc}
\usepackage{epstopdf}
\usepackage{epsfig}

\newtheorem{thm}{\bf{Theorem}}[section]
\newtheorem{lem}[thm]{\bf{Lemma}}
\newtheorem{note}[thm]{\bf{Note}}
\newtheorem{df}[thm]{\bf{Definition}}
\newtheorem{cor}[thm]{\bf{Corollary}}

\newtheorem{prop}[thm]{\bf{Proposition}}
\newtheorem{fact}[thm]{\bf{Fact}}
\newtheorem{ex}[thm]{\bf{Example}}

\numberwithin{equation}{section}

\newcommand{\dom}{\operatorname{dom}}

\newcommand{\epi}{\operatorname{epi}}

\newcommand{\trace}{\operatorname{trace}}
\newcommand{\lin}{\operatorname{lin}}

\newcommand{\Id}{\operatorname{Id}}

\newcommand{\sgn}{\operatorname{sgn}}

\newcommand{\conv}{\operatornamewithlimits{conv}}
\newcommand{\argmin}{\operatornamewithlimits{argmin}}

\newcommand{\ran}{\operatorname{ran}}

\newcommand{\ri}{\operatorname{ri}}

\newcommand{\Span}{\operatornamewithlimits{span}}

\newcommand{\R}{\operatorname{\mathbb{R}}}

\newcommand{\U}{\operatorname{\mathcal{U}}}
\newcommand{\M}{\operatorname{\mathcal{M}}}
\newcommand{\V}{\operatorname{\mathcal{V}}}

\newcommand{\VU}{\operatorname{\mathcal{VU}}}
\newcommand{\OU}{\operatorname{\overline{U}}}
\newcommand{\OV}{\operatorname{\overline{V}}}

\newcommand{\Proj}{\operatorname{Proj}}

\begin{document}

\title{The chain rule for $\V\U$-decompositions of nonsmooth functions}

\author{W. Hare\thanks{Mathematics, University of British Columbia Okanagan, Kelowna, B.C., V1V 1V7, Canada. Research by this author was partially supported by NSERC Discovery Grant \#2018-03865. warren.hare@ubc.ca}\and C. Planiden\thanks{Mathematics and Applied Statistics, University of Wollongong, Wollongong, NSW, 2522, Australia. Research by this author was supported by UBC UGF and by NSERC Canada. chayne@uow.edu.au}\and C. Sagastiz\'abal\thanks{
IMECC - UNICAMP, 13083-859, Campinas, SP, Brazil. Research by this author
was partially supported by CNPq Grant 303905/2015-8, by FAPERJ, and CEMEAI, Brazil.
sagastiz@unicamp.br}}
\maketitle

\begin{abstract}
In Variational Analysis, $\VU$-theory provides a set of tools that is helpful for understanding and exploiting the structure of nonsmooth functions. The theory takes advantage of the fact that at any point, the space can be separated into two orthogonal subspaces: one that describes the direction of nonsmoothness of the function, and the other on which the function behaves smoothly and has a gradient. For a composite function, this work establishes a chain rule that facilitates the computation of such gradients and characterizes the smooth subspace under reasonable conditions.  From the chain rule presented, formul\ae\/ for the separation, smooth perturbation and sum of functions are provided. Several nonsmooth examples are explored, including norm functions, max-of-quadratic functions and LASSO-type regularizations.
\end{abstract}

\section{Introduction}\label{sec:introduction}

Nonsmooth optimization methods face the challenge of slow convergence rates. When dealing with smooth functions, there are well-known minimization techniques that can achieve superlinear or quadratic convergence to a minimizer. In the nonsmooth setting, however, algorithms with that level of convergence speed remain elusive. There exist many approaches to nonsmooth minimization, such as proximal methods \cite{bello2011strongly,chen2012smoothing,kiwiel1985methods,nesterov2005smooth,Teboulle2018}, bundle methods \cite{belloni2009dynamic,hare2018computing,hou2008global,hareetal16}, trust-region methods \cite{akbari2015new,conn2000trust,desampaio1997trust,martinez1997trust}, conjugate gradient and gradient sampling methods \cite{burke2005robust,min2011linearly}, all of which have linear convergence at best \cite{karmitsa2012comparing,wen2017linear,xiao2015convergence}.\par 
A better understanding of the structure underlying a nonsmooth function is instrumental in improving this shortcoming. In this paper, we concentrate on a particular partitioning of the domain space called \emph{$\VU$-decomposition}; see \cite{sagastizabal-icm2018} and references therein. Given an objective function and a current point of interest, the $\VU$-decomposition splits the space into two orthogonal subspaces. In doing so, we may take advantage of the fact that locally the objective function is nonsmooth parallel to one of the subspaces (the $\V$-space), while parallel to the remaining subspace (the $\U$-space) the function is smooth. This allows the optimizer to exploit the smoothness in the $\U$-space and calculate useful objects such as the \emph{$\U$-Lagrangian} and the \emph{$\U$-Hessian}, which are defined and explained in detail in Section \ref{sec:vudecomposition}. The $\VU$-decomposition is used in a nonsmooth minimization algorithm called the $\VU$-algorithm, which has been proved superlinearly convergent in the convex case \cite{vualg}.\par Since its inception \cite{ulagconv,lemarechal-sagastizabal-1997}, $\VU$-theory has been explored and expanded in both the convex \cite{Hare2014,hare2019derivative,mifflin2000functions,vutheory,mifflin2003primal,vudecomp} and nonconvex \cite{proxbund,hareetal16,vusmoothness} settings. Of particular interest to the present work is the progress defining and working with the \emph{primal-dual gradient (PDG)} structured functions that have {\em fast tracks} \cite{mifflin2000functions,vutheory,mifflin2002proximal,mifflin2003primal}. Fast tracks provide structural information for PDG functions, even if strong transversality does not hold. These terms are defined formally in the next section; we mention here that the existence of a fast track is the property that allows the $\VU$-algorithm to identify points with favourable $\VU$-decompositions and thereby converge superlinearly \cite{liu-sagastizabal,mifflin2002proximal}.\par The goal of this paper is the advancement of $\VU$-theory of PDG functions with fast tracks, focusing on calculus of the $\VU$-decomposition and the gradient of the $\U$-Lagrangian. The calculus formul\ae\/ are derived for nonsmooth functions $f$ resulting from the composition of a $\mathcal{C}^2$ vector mapping $\Phi:\R^m\to\R^n$ with a convex function $h:\R^n\to\R$, i.e., $f=h\circ\Phi$. We establish a new equivalency between the gradient of the $\U$-Lagrangian and what we term the \emph{$\U$-gradient}, and we construct calculus rules for the $\U$-gradient. The general approach starts as in \cite{Hare06} which, after drawing a relation between $\VU$-structures and partly smooth functions, applies the chain rule in \cite{Lewis02}. For the considered setting, however, the development is not straightforward, as the aforementioned chain rule does not explore gradient structure.\par The main result of this work is Theorem \ref{thm:main}, providing expressions for both the $\U$-space and the $\U$-gradient of the function $f$ in terms of the $\VU$-decomposition of $h$ and the Jacobian of $\Phi$. Based on that result, we derive a separability rule, a smooth perturbation rule and a sum rule. Since our formul\ae\/ are obtained without assuming strong transversality, they can be applied to compute $\U$-gradients for $\ell_1$- regularized functions, including in particular the objective functions of the well-known LASSO problems. \par The remainder of this paper is organized as follows. The rest of Section \ref{sec:introduction} provides general notation used throughout. Section \ref{sec:vudecomposition} contains definitions of the relevant function classes and the $\VU$-decomposition objects, as well as the basics of $\VU$-theory. Section \ref{sec:partialsmoothness} outlines the important relationship between fast tracks and partly smooth functions. The difference between transversality and nondegeneracy of partly smooth functions is also discussed. Section \ref{sec:ugradient} shows how the gradient of the $\U$-Lagrangian and the $\U$-gradient are related and renders the chain rule for functions with fast tracks. In Section \ref{sec:examples}, we lay out the rules for the smooth perturbation and the sum of functions, and present a set of examples to illustrate those results; we explore convex finite-max functions, LASSO functions and $\ell_1$-regularized functions. Section \ref{sec:conclusion} makes some summarizing remarks and suggests avenues of future research in this area.

\subsection*{Notation}\label{sec:notation}

We generally use the notation of \cite{rockwets}. We denote $\R\cup\{+\infty\}$ by $\overline{\R}$. The identity matrix is denoted by $\Id$. The open ball of radius $\delta$ about the point $\bar{x}$ is denoted by $B_\delta(\bar{x})$. The domain and range of $f$ are denoted by $\dom f$ and $\ran f$, respectively. The indicator function of $S$ is denoted by $\iota_S(x)$. The projection mapping onto $S$ is defined by $\Proj_S(x)=\argmin_{s \in S}\{\|s - x\|\}$. The relative interior of $S$ is denoted by $\ri S$. The epigraph of $f$ is defined by $\epi f=\{(x,\alpha):\alpha\geq f(x)\}$.

\section{$\VU$-decomposition}\label{sec:vudecomposition}

\subsection{Primal-dual gradient structure and fast tracks}

The $\VU$-decomposition of $\R^n$ for a function $f$ at a point $\bar{x}$ was originally defined for $f$ convex \cite{ulagconv,lemarechal-sagastizabal-1997} and has since been generalized to lower semicontinuous (lsc) functions that have PDG structure and fast tracks \cite{vusmoothness}. These concepts and other pertinent terms are contained in this section.  To define PDG structures, we first recall the following definitions.
\begin{df}[Subgradients and Subdifferentials] Consider a function $f:\R^n\to\overline{\R}$ and a point $\bar{x}$ with $f(\bar{x})\in\R$. A vector $g\in\R^n$ is a
\begin{itemize}
\item[\rm(i)] \emph{regular subgradient} of $f$ at $\bar{x}$, written $g\in\hat\partial f(\bar{x})$, if
$$f(x)\geq f(\bar{x})+g^\top(x-\bar{x})+o(\|x-\bar{x}\|);$$
\item[\rm(ii)]\emph{(general) subgradient} of $f$ at $\bar{x}$, written $g\in\partial f(\bar{x})$, if there exist sequences $x_k\underset{f}\to\bar{x}$ and $g_k\to g$ with $g_k\in\hat\partial f(x_k)$;
\item[\rm(iii)]\emph{horizon subgradient} of $f$ at $\bar{x}$, written $g\in\partial^\infty f(\bar{x})$, if the same holds as in (ii) except that instead of $g_k\to g$, one has $\lambda_kg_k\to g$ for some sequence $\lambda_k\searrow0$.
\end{itemize}
The sets $\hat\partial f(\bar{x})$, $\partial f(\bar{x})$ and $\partial^\infty f(\bar{x})$ are called the \emph{regular subdifferential}, \emph{(limiting) subdifferential} and \emph{horizon subdifferential} of $f$ at $\bar{x}$, respectively.\end{df}
\begin{df}[Cones]Consider a set $S$ and a point $\bar{x}$;
\begin{itemize}
\item[\rm(i)] the \emph{tangent cone} to $S$ at $\bar{x}$ is defined by
$$T_S(\bar{x})=\limsup\limits_{\tau\searrow0}\frac{1}{\tau}(S-\bar{x});$$
\item[\rm(ii)] the \emph{regular normal cone} to $S$ at $\bar{x}$ is defined by
$$\widehat N_S(\bar{x})=\{g\in\R^n:g^\top(x-\bar{x})\leq o(\|x-\bar{x}\|)~\forall x\in S\};$$ 
\item[\rm(iii)] the \emph{normal cone} to $S$ at $\bar{x}$ is defined by $$N_S(\bar{x})=\{g\in\R^n:~\exists~x_k\underset S\to\bar{x},g_k\to g\mbox{ \emph{with} }g_k\in\hat N_S(\bar{x})\}.$$
\end{itemize}\end{df}
\begin{fact}[\cite{rockwets} Theorem 8.9]\label{fact1}
For $f:\R^n\to\R$ and any point $\bar{x}\in\dom f$, we have$$\partial^\infty f(\bar{x})\subseteq\{g\in\R^n:(g,0)\in N_{\epi f}(\bar{x},f(\bar{x}))\},$$and the inclusion is an equality whenever $f$ is lsc near $\bar{x}$.\end{fact}
We next define a $d$-dimensional $\mathcal{C}^2$-manifold.  Some literature refers to this as a submanifold, as it is embedded it in $\R^n$.  We use the term manifold in consistency with the notation of \cite{rockwets}.
\begin{df}[Manifold]
The set $\M\subseteq\R^n$ is a $d$-dimensional \emph{smooth $\mathcal{C}^2$-manifold} in $\R^n$ about the point $\bar{x}\in\M$ if $\M$ can be represented relative to an open neighbourhood $O(\bar{x})$ as the set of solution vectors to $F(x)=0$, where $F:O\to\R^m$ is a $\mathcal{C}^2$ mapping 
whose $m\times n$ Jacobian matrix $\nabla F(\bar{x})$ is surjective and has full rank $m=n-d$.\end{df}
\noindent When thinking of manifolds, it is useful to recall that the normal and tangent cones to manifolds are well-behaved subspaces. To that end, we remind the reader of the concept of a Clarke regular function.
\begin{df}[Clarke Regularity]A function $f:\R^n\to\overline{\R}$ is \emph{Clarke regular} at $\bar{x}$ if $$\widehat N_{\epi f}(\bar{x})=N_{\epi f}(\bar{x}).$$\end{df}
\begin{fact}\emph{\cite[Exercise 6.8]{rockwets}}
Let $\M\subseteq \R^n$ be a $d$-dimensional manifold about $\bar{x}$, with associated defining function $F:\R^n\to\overline{\R}$ such that $\nabla F(\bar{x})$ is of rank $m=n-d$. Then $\M$ is Clarke regular and geometrically derivable at $\bar{x}$, and the tangent and normal cones to $\M$ at $\bar{x}$ are linear subspaces orthogonally complementary to each other:
$$T_{\M}(\bar{x})=\{w\in\R^n:\nabla F(\bar{x})w=0\},\qquad N_{\M}(\bar{x})=\{\nabla F(\bar{x})^\top y:y\in\R^m\}.$$\end{fact}
\noindent We are now ready to define PDG structures.
\begin{df}[PDG structure]\label{df:pdg}
An lsc function $f:\R^n\to\overline{\R}$ has \emph{primal-dual gradient (PDG)} structure at a point $\bar{x}$ relative to the ${m_1+m_2}$-dimensional manifold $\M\subseteq \R^n$ if there exist functions $\{f_i\}_{i=0}^{m_1}$ and $\{\varphi_j\}_{j=1}^{m_2}$ that are $\mathcal{C}^2$ on a ball $B_\delta(\bar{x})$, and a closed convex set $\Delta\subseteq\R^{m_1+1+m_2}$, that locally satisfy
\begin{itemize}
\item[\rm(i)] $\bar{x}\in\{x\in B_\delta(\bar{x}):f_i(x)=f(x),i=0,\ldots,m_1;\varphi_j(x)=0,j=1,\ldots,m_2\}\subseteq\M\cap B_\delta(\bar{x})$;
\item[\rm(ii)]if $(\alpha,\beta)=(\alpha_0,\alpha_1,\ldots\alpha_{m_1},\beta_1,\beta_2,\ldots,\beta_{m_2})\in\Delta,$ then $\alpha$ is an element of the canonical simplex $\Delta_1:$
$$\alpha\in\Delta_1=\left\{(\alpha_0,\alpha_1,\ldots,\alpha_{m_1}):\sum\limits_{i=0}^{m_1}\alpha_i=1,\alpha_i\geq0\right\};$$
\item[\rm(iii)] the $i$\textsuperscript{th} canonical vector is in $\Delta$ for $i=0,1,\ldots,m_1$;
\item[\rm(iv)]for each $\bar{j}=1,2,\ldots,m_2,$ there exists $(\alpha,\beta)\in\Delta$ such that $\beta_{\bar{j}}\neq0$ and $\beta_j=0$ for $j\neq\bar{j}$;
\item[\rm(v)] for each $x\in\M\cap B_\delta(\bar{x})$, we have
\begin{itemize}
\item[\rm(a)] $f(x)=f_i(x)$ for some $i$, and
\item[\rm(b)] $g\in \partial f(x)$ if and only if
$$g=\sum\limits_{i=0}^{m_1}\alpha_i\nabla f_i(x)+\sum\limits_{j=1}^{m_2}\beta_j\nabla\varphi_j(x),$$
where $(\alpha,\beta)\in\Delta$ satisfies
$$\begin{cases}
\alpha_i=0&\mbox{if }f_i(x)\neq f(x),\\\beta_j=0&\mbox{if }\varphi_j(x)\neq0.
\end{cases}$$
\end{itemize}\end{itemize}\end{df}
\noindent In Definition \ref{df:pdg}, the functions $f_i$ and $\varphi_j$, and the set $\Delta$, account for the primal and dual structural information, respectively. The simplest instance of PDG structure is given by the scalar absolute value function $f(x)=|x|$ at $\bar x=0$, for which $f_0(x)=-f_1(x)=x$, there is no $\varphi$-function, and the dual set is the canonical simplex in $\R^2$. More elaborate examples are considered in Section \ref{ss-expdg} (and \cite{mifflin2000functions}).\par The  property of strong transversality given below is associated with particularly well-behaved PDG structured functions.
\begin{df}[strong transversality]\label{strongtrans}
For a PDG structured function, 
the collection of primal information, $f_i,\varphi_j:\R^n\to\overline{\R}$, 
given by index sets $i \in \{0, 1, ..., m_1\}$ and $j \in \{1, ..., m_2\}$ is \emph{strongly transversal} to the $(m_1+m_2)$-dimensional manifold $\mathcal{M}\subseteq\R^n$ if the $n\times(m_1+m_2)$ matrix
\begin{equation}\label{eq:st}
\overline{V}=\left[\{\nabla f_i(\bar{x})-\nabla f_0(\bar{x})\}_{i=1}^{m_1},\{\nabla\varphi_j(\bar{x})\}_{j=1}^{m_2}\right]\end{equation}
has full column rank.\end{df}
\noindent For the absolute-value function, $\overline{V}=\left[\{-1-1\}\right]=\left[-2\right]$ is trivially strongly transversal. We mention in passing that the PDG representation is not unique; the PDG construct in Example~\ref{ex:ell1reg}, given for $\ell_1$-regularized functions, provides an alternative structure for the absolute-value function that is not strongly transversal (the corresponding matrix $\overline{V}$ is $1\times 2$). Strong transversality is related to the linear independence of the gradients of the primal functions in \eqref{eq:st} and, hence, to the fact that the structure is defined without any redundant information (see the comments after Theorem \ref{thm:fast track}).

\subsubsection{Examples of PDG structure}\label{ss-expdg}

Two examples of common classes of functions that have PDG structure are convex finite-max functions and maximum eigenvalue functions \cite{mifflin2003primal}.  In this section, we quickly review these examples.  In addition, we examine the $\ell_1$-regularization problem. For the first example, we require the definition of active set.
\begin{df}[Active set]
Let $f:\R^n\to\overline{\R}$ be a finite-max function, i.e. the pointwise maximum of a finite set of $\mathcal{C}^2$ functions:$$f=\max_if_i,~f_i\in\mathcal{C}^2,~i=0,1,\ldots,p.$$ The \emph{active set} $A(\bar{x})$ of $f$ at $\bar{x}\in\dom f$ is the set of all subindices $i$ such that $f_i(\bar{x})=f(\bar{x})$.
\end{df}
\begin{note} Henceforth, we assume without loss of generality that $0\in A(\bar{x})$, reordering subindices if required.\end{note}
\begin{ex}[Finite-max]\label{ex:finitemax}
Let $f:\R^n\to\overline{\R}$ be a finite-max function. Then $f$ has PDG structure at any point $\bar{x}\in\R^n$. If, in addition, the set $\{\nabla f_i(\bar{x})-\nabla f_0(\bar{x}):i\in A(\bar{x})\setminus\{0\}\}$ is linearly independent, then the PDG structure of $f$ satisfies strong transversality at $\bar{x}$.
\end{ex}
\begin{proof}
Given any point $\bar{x}$ fixed, we make the following choices to show that $f$ is a PDG-structured function at $\bar{x}$ \cite{vutheory,mifflin2003primal}. Using the notation of Definition \ref{df:pdg}, we set $$m_2=0,\quad m_1+1=|A(\bar{x})|.$$Set $\Delta=\Delta_1$, and choose $\delta$ small enough that $B_\delta(\bar{x})$ excludes the set of functions $\{f_i:i\not\in A(\bar{x})\}$ from the local structure. We have that $f_i(\bar{x})=f(\bar{x})$ for all $i\in\{0,1,\ldots,m_1\}$ by the definition of active set, and there are no $\varphi_j$ functions since $m_2=0$. Thus, point (i) of Definition \ref{df:pdg} is satisfied. Points (ii), (iii) and (iv) are also immediately satisfied, since $m_2=0$. For point (v)(a), since $\delta$ is small enough to exclude the inactive functions, we have that for each $x\in\M\cap B_\delta(\bar{x})$ there exists $i\in\{0,\ldots,m_1\}$ such that $f(x)=f_i(x)$. For point (v)(b), since $f$ is a finite-max function, the subgradients of $f$ at $\bar{x}$ have the form $$g=\sum\limits_{i=0}^{m_1}\alpha_i\nabla f_i(\bar{x})\qquad\mbox{ \cite[Exercise 8.31]{rockwets}}.$$ Therefore, $f$ has PDG structure. Finally, if $\left\{\nabla f_i(\bar{x})-\nabla f_0(\bar{x}):i\in A(\bar{x})\setminus\{0\}\right\}$ is a linearly independent set, then the matrix of \eqref{eq:st} is $[\{\nabla f_i(\bar{x})-\nabla f_0(\bar{x})\}_{i=1}^{m_1}]$ and has full column rank \cite[\S 4.1]{mifflin2003primal}. Therefore, strong transversality is satisfied.\end{proof}
\noindent Since the finite-max example has no $\varphi$-functions, the dual set $\Delta$ coincides with the canonical simplex. Our next example deals with a more convoluted dual set.
\begin{ex}[Maximum eigenvalue] Let $A(\cdot)$ be an $m\times m$ symmetric matrix function whose elements are $\mathcal{C}^2$ functions on $\R^n$. Define $f:\R^n\to\overline{\R}$,$$f(x)=\max\limits_{t\in\mathcal{T}}\{F(x,t)\},$$
where $F(x,t)=t^\top A(x)t\mbox{ and }\mathcal{T}=\{t\in\R^m:t^\top t=1\}.$ It is known that $f(x)$ is the maximum eigenvalue of $A(x)$. Suppose that $A$ is such that $f$ is convex on $\R^n$. Then $f$ has PDG structure at any point $\bar{x}\in\R^n$.\end{ex}
\begin{proof}
The proof of this example is much more involved than that of the previous one, so we give an overview here and refer the reader to \cite[\S 3.2]{vutheory} for the details. Using the Frobenius inner product $\langle P,Q\rangle=\trace(PQ)$ on the space $\mathcal{S}$ of $s\times s$ symmetric matrices, we suppose that $f(\bar{x})$ has multiplicity $s$ and that the first eigenspace $\mathcal{E}^1(\bar{x})$ has basis matrix $$E^1(\bar{x})=[e_1(\bar{x})~~e_2(\bar{x})~~\cdots~~e_s(\bar{x})].$$Then by \cite[Theorem 3]{overton1992large}, the subgradients of $f$ at $\bar{x}$ have the form
$$g=(g_1,\ldots,g_n)^\top\in\partial f(\bar{x})\Leftrightarrow g_j=\left\langle S,E^1(\bar{x})^\top\frac{\partial A(\bar{x})}{\partial x_j}E^1(\bar{x})\right\rangle,j=1,\ldots,n,$$
where $S\in\Delta.$ For $\Delta$, we use the set of $s\times s$ dual feasible matrices:
$$\Delta=\{S\in\mathcal{S}:S\mbox{ is positive semidefinite and }\trace(S)=1\}.$$ This choice is shown in \cite[\S 3.2]{vutheory} to satisfy point (iii) of Definition \ref{df:pdg}. To define the $f_i$ and $\varphi_j$ functions, we define
$$I_1=\{1,2,\ldots,s\}\qquad\mbox{ and }\qquad I_2=\{(k,l)\in I_1\times I_1:k<l\}.$$Then $|I_1|=s,$ $|I_2|=s(s-1)/2$ and $|I_1|+|I_2|=s(s+1)/2$. We have continuity of eigenvalues of $A$, so there exists $\varepsilon>0$ such that for each $x\in B_\varepsilon(\bar{x})$, the multiplicity of $f(x)$ is at most $s$. By \cite[pp. 557--559]{wilkinson1965algebraic}, there exist $s$ $\mathcal{C}^2$ functions $q_i:B_\varepsilon(\bar{x})\to\R^m$, $i\in I_1$, that satisfy
\begin{align*}
q_i(\bar{x})^\top A(\bar{x})q_i(\bar{x})&=f(\bar{x}),&&\mbox{ for }i\in I_1,\\
q_k(\bar{x})^\top A(\bar{x})q_l(\bar{x})&=0,&&\mbox{ for }(k,l)\in I_2,\mbox{ and}\\
q_k(x)^\top q_l(x)&=\delta_{kl},&&\mbox{ for }(k,l)\in I_2, \forall x\in B_\varepsilon(\bar{x}),
\end{align*}
where $\delta_{ii}=1$ and $\delta_{kl}=0$ for $k\neq l$. Then $\{q_i(\bar{x})\}_{i\in I_1}$ is an orthonormal basis of eigenvectors for $\mathcal{E}^1(\bar{x}).$ We define
$$\phi_{kl}(x)=q_k(x)^\top A(x)q_l(x)$$ and set
$$\mathcal{M}=\{x\in B_\varepsilon(\bar{x}):\phi_{kl}(x)=0~\forall(k,l)\in I_2\}.$$Then for any $x\in\mathcal{M}$, $i\in I_1$  and $(k,l)\in I_2$, we have
\begin{align*}
\phi_{kl}(x)\delta_{kl}&=q_k(x)^\top A(x)q_l(x),\\
A(x)q_i(x)&=\phi_{ii}(x)q_i(x),\\
f(x)&=\max\limits_{j\in I_1}\phi_{jj}(x).
\end{align*}
Therefore, setting
\begin{align*}
m_1&=s-1,&f_{i-1}&=\phi_{ii}\mbox{ for }i\in I_1,\\
m_2&=s(s-1)/2,&\varphi_j&=\phi_{kl}\mbox{ for }(k,l)\in I_2,
\end{align*}we have that points (i), (ii) and (iv) of Definition \ref{df:pdg} are satisfied \cite[\S 3.2]{vutheory}. Only point (v) remains, for which we express $\partial f(x)$ in terms of $\partial\phi_{kl}$. Using \cite[Lemma 3.3]{vutheory}, we find that every $g\in\partial f(\bar{x})$ is a linear combination of the set$$\left\{\nabla\phi_{kl}(x)=\left(\frac{\partial\phi_{kl}(x)}{\partial x_1},\ldots,\frac{\partial\phi_{kl}(x)}{\partial x_n}\right)^\top\right\}_{(k,l)\in I_1(x)\times I_1(x)},$$where the multipliers in the linear combination form a matrix $S\in\mathcal{S}$. Denoting by $s_{ip}$ the element of row $i$, column $p$ of $S$, the choice of $\alpha_i$ and $\beta_j$ that satisfies point (v) is
\begin{equation*}\alpha_{i-1}=\begin{cases}s_{ii},&\mbox{ if }i\in I_1(x),\\0,&\mbox{ if }i\in I_1\setminus I_1(x),\end{cases}\mbox{ and }\beta_j=\begin{cases}2s_{kl},&\mbox{ if }(k,l)\in I_2(x),\\0,&\mbox{ if }(k,l)\in I_2\setminus I_2(x).\end{cases}\qedhere\end{equation*}\end{proof}
Our final example gives a good illustration of the interest of considering PDG structures with non-null $\varphi$-functions, therefore yielding dual sets $\Delta$ different from the canonical simplex $\Delta_1$.
\begin{ex}[$\ell_1$-regularization]\label{ex:ell1reg}
Let $f:\R^n\to\overline{\R}$ be $\mathcal{C}^2$, $\|x\|_1 = \displaystyle\sum_{i=1}^n |x_i|$ and $\tau>0$.  Consider the $\ell_1$-regularization problem:
$$\min_{x \in \R^n} \{ f(x) + \tau\|x\|_1 \}.$$\end{ex}
\noindent The minimand $f(x) + \tau\|x\|_1$ can be written as the maximum of a finite number of smooth functions, so the approach of Example \ref{ex:finitemax} could be applied.  However, to account for the sign change in each component of $x$, the required number of subfunctions is $2^{n}$, which is clearly undesirable. In order to acquire a more
succinct PDG structure for the $\ell_1$-regularization problem, we begin by thinking about the equivalent problem:
$$\min_{r,x\in \R^n} \{ f(x) + \tau\|r\|_1 : r=x\}.$$
While this rewriting doubles the number of variables, it also allows for a PDG structure without resorting to the finite-max framework.  We note that this reformulation is used in several algorithms, such as applying ADMM to the LASSO problem \cite{Boydetal2011, eckstein1992douglas}. Let $F(r,x) = f(x) + \tau\|r\|_1$ and fix a point $(\bar{r},\bar{x})$ with $\bar{r}=\bar{x}$. Then the active set is$$A(\bar{r},\bar{x}) = \{ i : \bar{r}_i = 0\},$$and the desired manifold is 
\begin{align*}
\M&=\{(r,x):r=x, A(r,x)=A(\bar{r},\bar{x})\}\\
&=\{(r,x):r=x, r_i=0\Leftrightarrow\bar{r}_i=0\}.
\end{align*}
\noindent Let $m_1=0$ and define 
$$f_0(r,x) = f(x) + \tau\sum_{i=1}^n\sgn(\bar{r}_i) r_i, ~~\mbox{where}~\sgn(\bar{r}_i) = \left\{\begin{array}{rl} -1, & \mbox{if }\bar{r}_i < 0, \\ 0, & \mbox{if }\bar{r}_i=0, \\ 1, & \mbox{if }\bar{r}_i > 0.\end{array}\right.$$ 
Let $m_2=|A(\bar{r},\bar{x})|+1$ and define 
\begin{align*}
\varphi_i(r,x)&=r_i, \quad i \in A(\bar{r},\bar{x}),\\
\varphi_{|A(\bar{r},\bar{x})|+1}(r,x)&=\|r-x\|^2.
\end{align*}
Finally, define
$$\Delta = \{(\alpha, \beta) \in \R \times \R^{m_2} : \alpha=1, \beta = [-1,1]^{m_2} \}.$$
\noindent We show that the above provides the PDG structure for $F$ at $(\bar{r},\bar{x})$ relative to $\M$. However, the PDG structure is not strongly transversal. Conditions (i), (ii), (iii), (iv), and (v)(a) of Definition \ref{df:pdg} are trivially true; we have only to prove condition (v)(b).  Considering $(r,x) \in M$, we find that 
$$\begin{array}{rcl}
\partial F(r,x) &=& \partial\left(f(x) + \|r\|_1\right)
= \left[\begin{array}{c} \gamma\tau \\ \nabla f(x)\end{array}\right],
\end{array}\mbox{ where }\gamma_i = \begin{cases}
\sgn(r_i), & \mbox{if }r_i \neq 0,\\
\ [-1,1], &  \mbox{if }r_i = 0.
\end{cases}$$
Conversely, for $(r,x) \in\M$, the set of $g$ defined by (v)(b) is
$$\begin{array}{rcl}
&& \left\{ g = \alpha\nabla f_0(r,x) + \sum_{j=1}^{m_2} \beta_j \nabla\varphi_j(r,x) : (\alpha,\beta) \in \Delta, \beta_j = 0 ~\mbox{if}~ \varphi_j(x) \neq 0\right\},\\
&=&\left\{ g = \left[\begin{array}{c}\tau\sgn(r) \\ \nabla f(x) \end{array}\right] 
+ \left[\begin{array}{c}\hat{\beta} \\ 0 \end{array}\right]  : 
\hat{\beta}_i = 0 ~\mbox{if}~ \varphi_j(x) \neq 0,  \hat{\beta}_i = [-1, 1] ~\mbox{if}~ \varphi_j(x) = 0 
\right\},\\
&=&\left[\begin{array}{c} \gamma\tau \\ \nabla f(x)\end{array}\right].
\end{array}$$Thus, condition (v)(b) holds and the PDG structure is proved. Now, considering $\OV$ as defined in Definition \ref{strongtrans}, we have that
$$\begin{array}{rcl}
\overline{V}&=&\left[\{\nabla f_i(\bar{r}, \bar{x})-\nabla f_0(\bar{r}, \bar{x})\}_{i=1}^{m_1}\cup\{\nabla\varphi_j(\bar{r}, \bar{x})\}_{j=1}^{m_2}\right],\\
&=&\left[\{\emptyset\} \cup \{\nabla\varphi_j(\bar{r}, \bar{x})\}_{j=1}^{m_2}\right],\\
&=&\left[ \left\{\left[\begin{array}{c} e_{i}\\0\end{array}\right]\right\}_{i \in A(\bar{r},\bar{x})} \cup \left[\begin{array}{c}2(\bar{r}-\bar{x}) \\ 2(\bar{x}-\bar{r})\end{array}\right]\right],
\end{array}$$
where $e_{i}$ is the $i$\textsuperscript{th} canonical vector.  Noting that $\bar{x}=\bar{r}$, we conclude that the PDG structure is not strongly transversal. Notwithstanding, this particular PDG structure will be useful for exhibiting a fast track for the $\ell_1$-regularized functions and, hence, applying our new chain rule; see Example~\ref{ex:ell1fasttrack}.

\subsection{$\VU$-structure}

The principle behind $\VU$-decomposition is that a nonsmooth, lsc function owes its nonsmoothness to a subspace only (the $\V$-space) and behaves smoothly on the remaining orthogonal subspace (the $\U$-space). The direct sum of these two subspaces is $\R^n$. We denote by $\OV\in\R^{n\times v}$ a basis matrix for the $\V$-space and by $\OU\in\R^{n\times u}$ a semiorthonormal (definition follows) basis matrix for the $\U$-space.
\begin{df}[Semiorthonormal]\label{df:semi}A matrix $A\in\R^{r\times c}$ is \emph{semiorthonormal} if either $c\geq r$ and the rows are orthonormal vectors (equivalently $AA^\top=\Id_r$), or $r>c$ and the columns are orthonormal vectors (equivalently $A^\top A=\Id_c$). In the case of a square matrix ($c=r$), semiorthonormality is equivalent to orthonormality.\end{df}
\begin{df}[Restriction] Given $\OV$ and $\OU$, the \emph{restriction} of any $x\in\R^n$ to the $\V$-space is defined by$$x_{\V}=\left(\OV^\top\OV\right)^{-1}\OV^\top x.$$Similarly, the restriction of $x$ to the $\U$-space is defined by$$x_{\U}=\OU^\top x.$$\end{df}
\noindent Note that $x_{\V}\in\R^{\dim\V}$ and $x_{\U}\in\R^{\dim\U}$. The inverse of $\OU^\top\OU$ is not needed in the definition of $x_{\U}$, because $\OU^\top\OU=\Id$ by Definition \ref{df:semi}.
In \cite{vusmoothness}, $x_{\V}$ and $x_{\U}$ are referred to as the \emph{projections} of $x\in\R^n$ onto the $\V$-space and $\U$-space, respectively. However, the projection as defined in this article is the orthogonal projection of $x$ onto a set and yields another vector in $\R^n,$ whereas $x_{\V}\in\R^v$ and $x_{\U}\in\R^u$. Therefore, we refer to $x_{\V}$ and $x_{\U}$ as restrictions rather than projections. In fact, $\Proj_{\V}x=\OV x_{\V}$ and $\Proj_{\U}x=\OU x_{\U}$, so one may view the orthogonal projection of $x$ onto the $\V$-space as the \emph{orthogonal lifting} of $x_{\V}$ into $\R^n$, and similarly for the $\U$-space and $x_{\U}$.\par Every $x\in\R^n$ is uniquely expressible in terms of its restrictions $x_{\V}$ and $x_{\U}$ \cite{vusmoothness}. Specifically,
$$x=\Proj_{\V}x+\Proj_{\U}x=\OV x_{\V}+\OU x_{\U}=\OV([\OV^\top\OV]^{-1}\OV^\top x)+\OU(\OU^\top x).$$
The separation of $\R^n$ into the $\V$-space and $\U$-space depends on the point of interest $\bar{x}\in\R^n$ and is achieved as follows.
\begin{df}[$\VU$-decomposition] \label{df:UV-decom}
Let $f:\R^n\to\overline{\R}$ be an lsc function and $\bar{x} \in \dom f$ with $\partial f(\bar{x})\neq\emptyset.$ Let $\bar{g} \in\ri\bar{\partial} f(\bar{x})$. The \emph{$\VU$-decomposition} of $\R^n$ for $f$ at $\bar{x}$ is defined by the subspaces
$$\V(\bar{x})=\mathrm{span}(\partial f(\bar{x})-\bar{g})\qquad\mbox{ \emph{and} }\qquad\U(\bar{x}) = N_{\partial f(\bar{x})}(\bar{g}).$$
Note that since $\bar g\in\ri\partial f(\bar{x})$, the normal cone defining $\U(\bar x)$ is a subspace.
\end{df}
\noindent Henceforth, the dependence of the subspaces on $f$ and $\bar x$ is omitted, unless needed for clarity.
The $\VU$-decomposition is independent of the choice of $\bar{g}\in\ri\partial f(\bar{x})$ \cite{vusmoothness}. The $\U$-restriction of $\bar{g}$ is the same as that of any other subgradient of $f$ at $\bar{x}$:
\begin{equation}\label{eq:w}
\bar{g}_{\U}=\OU^\top\bar{g}=\OU^\top g\quad\mbox{ for any }g\in\bar{\partial}f(\bar{x}).
\end{equation}
For a function $f$ that has PDG structure, the $\V$-space at $\bar{x}\in\dom f$ can be expressed in terms of the primal function gradients \cite{vusmoothness}:
\begin{align}
\V&=\lin\left[\{\nabla f_i(\bar{x})-\nabla f_0(\bar{x})\}_{i=0}^{m_1},\{\nabla\varphi_j(\bar{x})\}_{j=1}^{m_2}\right].\label{eq:vspace}
\end{align}
If, in addition, $f$ satisfies strong transversality, then the matrix $\overline{V}$ defined in \eqref{eq:st} is a basis matrix for $\V$ (if strong transversality does not hold, a subset
of the index sets \(\{0,1,\ldots,m_1\}\) and \(\{1,\ldots,m_2\}\) defines a basis matrix for
$\V$; see the comments after Theorem \ref{thm:fast track}).
\begin{df}[$\U$-Lagrangian] \label{df:U-Lag}
Let $f:\R^n\to\overline{\R}$ be a PDG function at $\bar{x} \in \dom f$ relative to the $d$-dimensional manifold $\M\subseteq \R^n$. Let $\bar{g} \in \ri\partial f(\bar{x})$.  The {\em $\U$-Lagrangian} of $f$ at $\bar{x}$ is defined by
$$L_{\U}f(u;\bar{g}) = \min_{v \in \V} \{f(\bar{x} + (\OU u+\OV v)) -\bar{g}^\top\OV v\},$$
where $\U$ and $\V$ are the $\VU$-decomposition subspaces and $\dim\V=n-d$. The related solution mapping is denoted 
$$W_{\U}f(u;\bar{g}) = \arg\min_{v \in \V} \{f(\bar{x} + (\OU u+\OV v)) -\bar{g}^\top\OV v\}.$$
\end{df}
\noindent A fundamental benefit of $\VU$-decomposition is that the gradient of the $\U$-Lagrangian of the objective function exists at the origin, even though the gradient of the objective function itself may not. This allows the application of gradient-based methods to the $\U$-Lagrangian.  Under favourable conditions, the $\U$-Lagrangian may even have a second-order expansion at the origin, which allows for quasi-Newton methods to be applied to the $\U$-Lagrangian. Since the $\U$-Lagrangian is a re-parameterization of $f$ along the $\U$-subspace, an algorithm designed to drive both $u$ and $\bar g$ to zero along iterations converges quickly, thus justifying the name ``fast'' track (see item (iv) in Theorem~\ref{thm:fast track}).
\begin{df}[Fast track] \label{df:fasttrack}
Let $f:\R^n\to\overline{\R}$ be a PDG function at $\bar{x} \in \dom f$ relative to the $d$-dimensional manifold $\M\subseteq\R^n$.  Let $\U$ and $\V$ be the $\VU$-decomposition subspaces, with basis matrices $\overline{U}$ and $\overline{V}$ and $\dim\V=n-d$.   Suppose that $\chi(u) = \bar{x} + (\overline{U} u + \overline{V} v(u))$, 
where $v: \U \mapsto \V$. The function $\chi$ is a fast track of $f$ at $\bar{x}$ for $\M$ if for any $\bar{g} \in \ri\partial f(\bar{x})$,
\begin{itemize}
\item[\rm(i)] $v$ is a $\mathcal{C}^2$ selection of $W_{\U}f(u;\bar{g})$ (i.e., $v(u) \in W_{\U}f(u;\bar{g})$, $v \in \mathcal{C}^2$) and
\item[\rm(ii)] $L_{\U}f(u;\bar{g})$ is $\mathcal{C}^2$ in $u$.
\end{itemize}
\end{df}
\noindent The fast track is nothing but a special re-parameterization of certain $\V$-components in terms of the respective $\U$-component. For PDG structured functions satisfying strong transversality, the result below gives a constructive expression for the fast track, based on the gradients of the primal functions.
\begin{thm}\emph{\cite[Theorem 3.1]{vusmoothness}}\label{thm:fast track}
Let $f:\R^n\to\overline{\R}$ be a PDG function that satisfies strong transversality at $\bar{x}$ relative to the $d$-dimensional manifold $\M\subseteq \R^n,$ and suppose that $$\dim\V=n-d\geq1,\dim\U\geq1.$$ Then for all $u$ small enough, the following hold.
\begin{itemize}
\item[\rm(i)] The nonlinear system with variable $v$ and parameter $u,$
\begin{align*}
f_i(\bar{x}+\overline{U}u+\overline{V}v)-f_0(\bar{x}+\overline{U}u+\overline{V}v)=0,&~i=1,\ldots,m_1,\\
\varphi_j(\bar{x}+\overline{U}u+\overline{V}v)=0,&~j=1,\ldots,m_2,\end{align*}
has a unique solution $v=v(u)$  such that
$v:\R^{\dim\U}\rightarrow\R^{\dim\V}$ is a $\mathcal{C}^2$ function
 satisfying $v(0)=0$. 
\item[\rm(ii)] The trajectory $\chi(u)=\bar{x}+\overline{U}u+\overline{V}v(u)$ has a $\mathcal{C}^1$ Jacobian:
$$\nabla\chi(u)=\overline{U}+\nabla v(u)=\overline{U}-\overline{V}(V(u)^\top\overline{V})^{-1}V(u)^\top\overline{U},$$
where
$$V(u)=\left[\left\{\nabla f_i(\chi(u))-\nabla f_0(\chi(u))\right\}_{i=1}^{m_1},\left\{\nabla\varphi_j(\chi(u))\right\}_{j=1}^{m_2}\right].$$
\item[\rm(iii)] In particular, $v(0)=0,\chi(0)=\bar{x},V(0)=\overline{V},\nabla v(0)=0,$ and $\nabla \chi(0)=\overline{U}.$
\item[\rm(iv)] The trajectory $\chi(u)$ is tangent to $\U$ at $\chi(0)=\bar{x},$ with $v(u)=O(\|u\|^2).$
\item[\rm(v)] The function $f(\chi(u))=f_i(\chi(u))$ for $i=0,1,\ldots,m_1,$ and $\chi(u)\in\M.$
\item[\rm(vi)] The matrix $V(u)\in\R^{n\times\dim\V}$ is a basis for $\V(u),$ and the matrix $\nabla \chi(u)\in\R^{n\times\dim\U}$, $$\nabla \chi(u)=\overline{U}+\overline{V}\nabla v(u),$$ is a basis for $\U(u).$
\end{itemize}
\end{thm}
\noindent In Theorem \ref{thm:fast track}, strong transversality is used to apply a second-order implicit function theorem and give a constructive expression $v=v(u)$. Strong transversality simplifies the presentation, but it is not a necessary condition for the existence of fast tracks. It is shown in \cite{mifflin2003primal} that for a PDG structured function to admit a fast track, it is sufficient to select a suitable subset of primal functions, eliminating redundant information. Specifically, consider $K=K_f\cup K_\varphi$ in the primal gradient index set, with
\(0\in K_f\subset\{0,1,\ldots,m_1\}\) and \(K_\varphi\subset\{1,\ldots,m_2\},\)
and suppose that 
\begin{itemize}
\item[(i)] the subspace $\V$ in \eqref{eq:vspace} is spanned by the reduced subset of indices
\[
\V=\V_K:=\lin\left[\{\nabla f_i(\bar{x})-\nabla f_0(\bar{x})\}_{i\in K_f},\{\nabla\varphi_j(\bar{x})\}_{j\in K_\varphi}\right]\,,\mbox{ and}
\]
\item[(ii)] the set of primal gradients above is linearly independent.
\end{itemize}
Then all the statements in Theorem~\ref{thm:fast track} hold, replacing \(\{0,1,\ldots,m_1\}\) by $K_f$, \(\{1,\ldots,m_2\}\) by $K_\varphi$ and $v=v(u)$ by $v_K=v_K(u)$
(see \cite[Theorem 4.2]{mifflin2003primal}). A particular case is the PDG structure presented in Example \ref{ex:ell1reg} for 
the $\ell_1$-regularization function, whose primal index set does not satisfy strong transversality.  However, the considered structure does admit a fast track at every point, a fact we will illustrate in Example \ref{ex:ell1fasttrack} once we have established our
chain rule and derived a formula for the sum of functions in 
Theorem~\ref{ex:sum}.\par The gradient of the $\U$-Lagrangian is closely related to the \emph{$\U$-gradient} of $f$, an object defined in Section \ref{sec:ugradient} that is the main focus of this paper. We explore properties of the $\U$-gradient and present the mathematical tools required for calculating the $\U$-gradient and $\U$-space of compositions of well-behaved functions. By ``well-behaved'', we mean that $f = h \circ \Phi$ where $h$ is PDG and has a fast track, and $\Phi$ is transversal to the fast track. We develop a chain rule, which allows for the computation of the $\U$-gradient of $f$ based on the analytic components of $h$ and $\Phi$.  In order to proceed, we need a slight divergence into {\em partly smooth} functions.

\section{Partial Smoothness}\label{sec:partialsmoothness}

PDG functions with fast tracks are closely related to partly smooth functions \cite{Hare06}. There are many useful properties of partly smooth functions found in \cite{Lewis02}, which we use to draw our conclusions about PDG functions in Sections \ref{sec:ugradient} and \ref{sec:examples}. In this section, we showcase the relationship and discuss the role that transversality (not to be confused with strong transversality) and nondegeneracy have to play in such results.
\begin{df}[partial smoothness] \label{df:PS}
 A function $f:\R^n\to\overline{\R}$ is {\em partly smooth} at a point $\bar{x}$ relative to a set $\M\ni\bar x$ if $\M \subseteq \R^n$ is an
$(n-\dim\V)$-dimensional manifold about $\bar{x}$ and
\begin{enumerate}
  \item[\rm(i)]  \emph{(smoothness)} $f$ restricted to $\M$ is a $\mathcal{C}^2$ function near $\bar x$;
  \item[\rm(ii)]  \emph{(Clarke regularity)} $f$ is Clarke regular at all points $x \in \M$ near $\bar x$, with $\partial f(x) \ne \emptyset$;
  \item[\rm(iii)]  \emph{(sharpness)} the affine span of $\partial f(\bar{x})$ (which is convex due to \emph{(ii)}) is a translate of $N_{\M} (\bar{x})$;
  \item[\rm(iv)]  \emph{(subcontinuity)} $\partial f$ restricted to $\M$ is continuous at $\bar{x}$.
\end{enumerate}
In this case, we refer to $\M$ as the {\em active manifold of partial smoothness}.
\end{df}
\noindent By the sharpness condition in (iii) above, the $\U$-subspace in Definition \ref{df:UV-decom} is the subspace tangent to $\M$ at $\bar x$. The relation between the manifold and a fast track is stated in the next theorem.
\begin{thm} \label{thm:FT=PS}
Let $f:\R^n\to\overline{\R}$ be a PDG function at $\bar{x}\in\dom f$ relative to $(n-\dim\V)$-dimensional manifold $\M\subseteq\R^n$. The active manifold of the partly smooth function and the fast track of $\VU$-theory have a one-to-one correspondence as follows.
\begin{enumerate}
\item[\rm(i)] If $f$ is partly smooth at $\bar{x}$ relative to $\M$, then $\M$ defines a fast track $$\chi(u)=\bar{x}+(u+v(u))$$ for $f$ at $\bar{x}$. In this case, $\M$ is locally expressible in the form  $\M = \{\bar{x} + (u + v(u)) : u \in \U\}$.
\item[\rm(ii)]  If $\chi(u)=\bar{x}+(u+v(u))$ is a fast track for $f$ at $\bar{x}$, then $f$ is partly smooth at $\bar{x}$ relative to $\M= \{\bar{x} + (u+v(u)) : u \in \U \}$.
\end{enumerate}
\end{thm}
\begin{proof} Theorem 3.1 of \cite{Hare06} provides the same statements under the conditions that $f$ is convex and $\bar{x}$ is a minimizer of $f$.  However, the proof of \cite[Thm 3.1]{Hare06} does not use either of these conditions and is directly applicable here.  (The proof held these conditions since in \cite{Hare06}, which was based on \cite{mifflin2002proximal}, fast tracks were only defined for convex functions at a minimizer.) \end{proof}
The chain rule for partly smooth functions requires the following definition of transversal functions.
\begin{df}[transversality]\label{df:trans} Let $\Phi:\R^m \rightarrow \R^n$ be a $\mathcal{C}^2$ function, $\M\subseteq\R^m$ be a manifold and $\bar{x} \in \dom\Phi\cap \M$.  We say $\Phi$ is {\em transversal} to $\M$ at $\bar{x}$ if
$$\{ z \in N_{\M}(\bar{x}) :  \nabla \Phi(\bar{x})^\top  z = 0 \} = \{0\}.$$
\end{df}
\noindent Equivalently, $\Phi$ is {\em transversal} to $\M$ at $\bar{x}$ if
$$\ran(\nabla\Phi(\bar{x})) + T_{\M}(\bar{x}) = \R^n.$$
\begin{thm}\emph{\cite[Theorem 4.2]{Lewis02}\label{thm:PSChain}} Let $\Phi : \R^m \rightarrow \R^n$ be $\mathcal{C}^2$ and $\bar{x} \in \dom\Phi$.  Suppose that $h:\R^n\to\overline{\R}$ is partly smooth at $\Phi(\bar{x})$ relative to the manifold $\M\subseteq\R^n$ and $\Phi$ is transversal to $\M$ at $\bar{x}$.  Then $h \circ \Phi$ is partly smooth at $\bar{x}$ relative to the manifold $\Phi^{-1}(\M)\subseteq\R^m$.
\end{thm}
\noindent Translating to the language of fast tracks and $\VU$-decompositions, we have the following theorem.
\begin{thm}\label{thm:chainfasttrack} Let $\Phi : \R^m \rightarrow \R^n$ be $\mathcal{C}^2$ and $\bar{x} \in \dom\Phi$. Let $\U$ and $\V$ be the $\VU$-decomposition of $h:\R^n\to\overline{\R}$ at $\Phi(\bar{x})$. Suppose $v(u)$ is a fast track of $h$ at $\Phi(\bar{x})$, and $\Phi$ is transversal to the manifold $\M = \{\Phi(\bar{x}) + (u + v(u)) ~:~ u \in \U\}$. Then $\Phi^{-1}(\M)$ is a fast track for $h \circ \Phi$ at $\bar{x}.$
\end{thm}
\begin{proof} Since $v(u)$ is a fast track of $h$ at $\Phi(\bar{x})$, by Theorem \ref{thm:FT=PS}, we know that $h$ is partly smooth at $\Phi(\bar{x})$ relative to $\M$. Applying Theorem \ref{thm:PSChain}, we have that $h \circ \Phi$ is partly smooth at $\bar{x}$ relative to $\Phi^{-1}(\M)$. Using Theorem \ref{thm:FT=PS}, we return to fast tracks and have that $\Phi^{-1}(\M)$ defines a fast track for $h \circ \Phi$ at $\bar{x}$. \end{proof}

\subsection{Transversality and nondegeneracy}

The notion of transversality is found in the theory of partial smoothness \cite{Lewis02,mifflin2003primal}, whereas the notion of nondegeneracy (defined below) is prevalent in $\VU$-theory and other subspace projection frameworks \cite{burke1988identification,calamai1987projected,dunn1987convergence,flaam1992finite,shapiro2003class,hare2007identifying}. These two concepts have a close relationship that we lay out in this section (see Proposition \ref{prop:trans_nondegen}). They are not equivalent in general (see Example \ref{ex:nonnon}); Proposition \ref{transgengood} gives conditions under which equivalence holds.
\begin{df}[Nondegeneracy] Let $\Phi:\R^m\to\R^n$ be $\mathcal{C}^2,$ $\bar{x}\in\dom\Phi,$ $h:\R^n\to\overline{\R}$ be convex. We say that $h\circ\Phi$ is \emph{nondegenerate} at $\bar{x}$ if
$$\{z\in\partial^\infty h(\Phi(\bar{x})):\nabla\Phi(\bar{x})^\top z=0\}=\{0\}.$$
\end{df}
\noindent First, we show that transversality implies nondegeneracy. To do so, we require the following definition and results.
\begin{df}[Indication function] For a function $h:\R^n\to\overline{\R}$ and a set $S\subseteq\R^n$, we define the indication function
$$h_S(x)=\begin{cases}
h(x),&\mbox{\emph{if} }x\in S,\\\infty,&\mbox{\emph{if} }x\not\in S.
\end{cases}$$
\end{df}
\begin{lem}\label{lem:horizon}Suppose $h:\R^n\to\overline{\R}$ is partly smooth at $\bar{z}$ relative to the $(n-\dim\V)$-dimensional manifold $\M\subseteq\R^n$.
Then
	$$\partial^{\infty}h(\bar{z}) \subseteq \partial^{\infty}h_{\M}(\bar{z}) =  N_{\M}(\bar{z}).$$
\end{lem}
\begin{proof} 
Using Clarke regularity of $h$ and the inclusion $\epi h_{\M}\subseteq\epi h$, we see that $$N_{\epi h}(\bar{z},h(\bar{z}))=\widehat N_{\epi h}(\bar{z},h(\bar{z}))\subseteq\widehat N_{\epi h_{\M}}(\bar{z},h(\bar{z}))\subseteq N_{\epi h_{\M}}(\bar{z},h(\bar{z})).$$This and the lower semicontinuity of $h_{\M}$ (see Definition \ref{df:PS}(ii) and Fact \ref{fact1}) give us that$$\partial^\infty h(\bar{z})\subseteq\partial^\infty h_{\M}(\bar{z}).$$Since $h$ is partly smooth, there exists a function $h_0 \in \mathcal{C}^2$ such that $h(x) = h_0(x)$ for all $x \in \M$ and therefore $h_{\M} = h_0 + \iota_{\M}$ (where $\iota_{\M}$ is  the indicator function).  By \cite[Theorem 8.9]{rockwets},
	$$\partial^\infty h_{\M}(\bar{z}) = \partial^\infty \iota_{\M}(\bar{z}).$$
Since $\iota_{\M}$ is Clarke regular, \cite[Exercise 8.14]{rockwets} shows that $\partial^\infty \iota_{\M}(\bar{z})=\partial \iota_{\M}(\bar{z})= N_{\M}(\bar{z})$, which leads to the desired equality.
\end{proof}
\begin{prop}\label{prop:trans_nondegen}
Let $\Phi:\R^m\to\R^n$ be $\mathcal{C}^2,$ $\bar{x}\in\dom\Phi,$ $h:\R^n\to\overline{\R}$ be partly smooth at $\Phi(\bar{x})$ relative to manifold $\M\subseteq\R^m.$ If $\Phi$ is transversal to $\M$ at $\bar{x},$ then $h\circ\Phi$ is nondegenerate at $\bar{x}.$
\end{prop}
\begin{proof}
Since $\Phi$ is transversal to $\M$ at $\bar{x},$ we have
$$\{z\in N_{\M}(\bar{x}):\nabla\Phi(\bar{x})^\top z=0\}=\{0\}.$$Applying Lemma \ref{lem:horizon}, we have that $\partial^\infty h(\Phi(\bar{x}))\subseteq N_{\M}(\bar{x}).$ Thus,
\begin{equation*}
\{z\in\partial^\infty h(\Phi(\bar{x})):\nabla\Phi(\bar{x})^\top z=0\}=\{0\}.\qedhere\end{equation*}
\end{proof}
\noindent The following example shows that the converse of Proposition \ref{prop:trans_nondegen} does not hold in general, that is, nondegeneracy does not imply transversality.
\begin{ex}\label{ex:nonnon}
Define $\Phi:\R^2\to\R^2,$ $\Phi(x,y)=(x^2,y)$ and $h:\R^2\to\overline{\R},$ $h(x,y)=|x|+y^2$. Then $h\circ\Phi$ is nondegenerate at $(0,0)\in\dom\Phi\cap\M,$ but $\Phi$ is not transversal to $\M$ at $(0,0),$ where $\M=\{(0,y)\}$ is the manifold with respect to $h$ at $(0,0)$.\end{ex}
\begin{proof}Note that $h$ is convex, lsc and full-domain. Thus, $\partial^\infty h(x,y)=\{0\}$ for any $(x,y)\in\R^2$ \cite[Theorem 9.13]{rockwets}. Therefore, nondegeneracy holds at $(0,0)$:
$$\{z\in\partial^\infty h(\Phi(0,0)):\nabla\Phi(0,0)^\top z=0\}=\{0\}.$$
To prove nontransversality, we will show that there exists a nonzero $z\in N_{\M}(0,0)$ such that $\nabla\Phi(0,0)^\top z=0.$
Denoting $z$ by $(z_1,z_2),$ we set $\nabla\Phi(x,y)^\top z=0$:
\begin{align*}
\nabla\Phi(x,y)&=\left[\begin{array}{c c}2x&0\\0&1\end{array}\right]=\nabla\Phi(x,y)^\top 
\nabla\Phi(x,y)^\top z&=\left[\begin{array}{c c}2x&0\\0&1\end{array}\right]\left[\begin{array}{c}z_1\\z_2\end{array}\right]=\left[\begin{array}{c}2xz_1\\z_2\end{array}\right]=\left[\begin{array}{c}0\\0\end{array}\right].
\end{align*}
Since $N_{\M}(0,0)=\{(x,0):x\in\R\},$ we have that
$$\{z\in N_{\M}(0,0):\nabla\Phi(0,0)^\top z=0\}=\{(z_1,0):z_1\in\R\}.$$Therefore, transversality does not hold.\end{proof}
\noindent Example \ref{ex:nonnon} proves that nondegeneracy does not imply transversality in general. However, the following proposition provides conditions that transform Proposition \ref{prop:trans_nondegen} into an if-and-only-if statement.
\begin{prop}\label{transgengood}
Let $\Phi:\R^m\to\R^n$ be $\mathcal{C}^2,$ $\bar{x}\in\dom\Phi,$ $h:\R^n\to\overline{\R}$ be partly smooth at $\Phi(\bar{x})$ relative to manifold $\M\subseteq\R^m.$ Then $h_{\M}$ is nondegenerate at $\Phi(\bar{x})$ if and only if $\Phi$ is transversal to $\M$ at $\bar{x}.$
\end{prop}
\begin{proof} By \cite[Example 3.2]{Lewis02}, we have that $h_{\M}$ is partly smooth at $\bar{x}$ relative to $\M.$\\
  $(\Leftarrow)$ Applying Proposition \ref{prop:trans_nondegen}, we have that if $\Phi$ is transversal to $\M$ at $\bar{x},$ then $h_{\M}$ is nondegenerate at $\Phi(\bar{x}).$\\
  $(\Rightarrow)$ Suppose $h_{\M}$ is nondegenerate at $\Phi(\bar{x})$ relative to $\M.$ Then
$$\{z\in\partial^\infty h_{\M}(\Phi(\bar{x})):\nabla\Phi(\bar{x})^\top z=0\}=\{0\}.$$ Since $h_{\M}$ is partly smooth on $\M,$ $h_{\M}$ is lsc on $\M=\dom h_{\M}.$ Hence, by Lemma \ref{lem:horizon} we have $\partial^\infty h_{\M} = N_{\M}.$ Therefore,
$$\{z\in N_{\M}(\bar{x}):\nabla\Phi(\bar{x})^\top z=0\}=\{0\},$$which is the definition of transversality.
\end{proof}
\begin{cor}\label{cor:subdiffchain}Let the assumptions and notation of Theorem \ref{thm:chainfasttrack} hold.  Then the following nondegeneracy condition holds:
	$$\{ z  \in \partial^{\infty}h(\Phi(\bar{x})) : \nabla \Phi^\top  z = 0\} = \{0\}.$$
Consequently, for $f = h \circ \Phi$ we have
	\begin{equation}\label{eq:subdifferentialchainrule}\partial f (\bar{x}) = \nabla\Phi(\bar{x})^\top  \partial h(\Phi(\bar{x})).\end{equation}
\end{cor}
\begin{proof} By Theorem \ref{thm:FT=PS}, we know that $h$ is partly smooth at $\Phi(\bar{x})$ relative to $\M$.
That is, $\M$ is the active manifold of partial smoothness.  Since $\Phi$ is transversal to $\M$,
	$$\{ z \in N_{\M}(\bar{x}) :  \nabla \Phi(\bar{x})^\top  z = 0 \} = \{0\}.$$
Applying Lemma \ref{lem:horizon}, $\partial^{\infty}f(\bar{x}) \subseteq N_{\M}(\bar{x})$, so
	$$\{ z  \in \partial^{\infty}h(\Phi(\bar{x})) : \nabla \Phi(\bar x)^\top  z = 0\} = \{0\}.$$
The remainder of the proof now follows immediately from \cite[Theorem 10.6]{rockwets}, noting that $h$ partly smooth implies $h$ is Clarke regular.\end{proof}
\begin{note} If $h$ is convex, then the horizon subdifferential is $\{0\}$ and \eqref{eq:subdifferentialchainrule} always holds.\end{note}

\section{The $\U$-gradient}\label{sec:ugradient}

We have established sufficient background theory to present our main result. Recall that while the gradient of the $\U$-Lagrangian $\nabla L_{\U}$ is the object used in \cite{vusmoothness} and several other papers on $\VU$-theory, it is an object in $\R^u$, which is not always convenient. We prefer to work with the $n$-dimensional analogue, which we call the \emph{$\U$-gradient} of $f$. We remind the reader that the $\U$-Lagrangian (Definition \ref{df:U-Lag}) is independent of the choice of $\bar{g}\in\ri\partial f(\bar{x})$ (see \eqref{eq:w}).

\begin{df}[$\U$-gradient]\label{def:ulagugrad}
Given the gradient of the $\U$-Lagrangian of $f:\R^n\to\overline{\R}$ at $\bar{x}$,
denoted by $\nabla L_{\U} f(\bar x)$,
 and the $\U$-basis matrix $\OU$, the \emph{$\U$-gradient} of $f$ at $\bar{x}$ is the vector $\nabla_{\U}f(\bar{x})$ defined by
\begin{equation}\label{eq:ulagugrad}
\nabla_{\U}f(\bar{x})=\OU\nabla L_{\U}f(\bar{x}).
\end{equation}That is, $\nabla_{\U}f$ is $\nabla L_{\U}f$ orthogonally lifted into $\R^n$.
\end{df}
\begin{lem}\label{lem:ulagugrad}
Given a $\U$-basis matrix $\OU$, the gradient of the $\U$-Lagrangian of $f:\R^n\to\overline{\R}$ at $\bar{x}$ is the restriction of the $\U$-gradient of $f$ at $\bar{x}$ to the $\U$-space:
$$\nabla L_{\U}f=\OU^\top\nabla_{\U}f.$$
\end{lem}
\begin{proof}
The statement is proved by premultiplying both sides of \eqref{eq:ulagugrad} by $\OU^\top$ and noting that $\OU^\top\OU=\Id$ by Definition \ref{df:semi}.
\end{proof}
Corollary \ref{cor:subdiffchain} tells us that under the conditions of Theorem \ref{thm:chainfasttrack}, the transversality condition of \cite{Lewis02} is sufficient to ensure that the subdifferential chain rule  holds. This allows us to derive the formula for the $\U$-gradient in this circumstance.
\begin{thm} \label{thm:main} Let $\Phi : \R^m \rightarrow \R^n$ be $\mathcal{C}^2$ and $\bar{x} \in \dom\Phi$. Let $\U$ and $\V$ be the $\VU$-decomposition of $h:\R^n\to\overline{\R}$ at $\Phi(\bar{x})$. Suppose $\chi(u) = \bar{x} + (u + v(u))$ is a fast track of $h$ at $\Phi(\bar{x})$, and $\Phi$ is transversal to the manifold $\M = \{\Phi(\bar{x}) + (u + v(u)) ~:~ u \in \U\}$. Then $f = h \circ \Phi$ has a fast track at $\bar{x}$.  Moreover, the $\U$-space and $\U$-gradient of $f$ at $\bar{x}$ can be computed as follows.
Select any $\bar{g} \in \ri\partial h(\Phi(\bar{x}))$. Then
\begin{equation}\label{eq:Uspace}\U = \{d \in \R^m : d^\top\nabla \Phi(\bar{x})^\top g = d^\top\nabla\Phi(\bar{x})^\top  \bar{g} \: \: \mathrm{for \: all} \: g \in \partial h(\Phi(\bar{x})) \}.\end{equation}
Consequently,
\begin{equation}\label{eq:Ugrad}\nabla_{\U}f(0)=\OU\OU^\top\nabla\Phi(\bar{x})^\top \bar{g},\end{equation}
where $\OU$ is a $\U$-basis matrix for $h$ at $\Phi(\bar{x})$.
\end{thm}
\begin{proof} Theorem \ref{thm:chainfasttrack} shows that $f$ has a fast track. Next, note that $\bar{g} \in \ri\partial h(\Phi(\bar{x}))$ implies
$$\nabla \Phi(\bar{x})^\top  \bar{g} \in \nabla \Phi(\bar{x})^\top  \ri(\partial h(\Phi(\bar{x}))) = \ri\left(\nabla \Phi(\bar{x})^\top  \partial h(\Phi(\bar{x}))\right) = \ri\partial f (\bar{x}).$$
The formula for $\U$ now follows from the characterization of the $\U$-subspace as the normal cone in Definition \ref{df:UV-decom}:
\[\U=\{d\in\R^n:d^\top g=
d^\top\bar{g}~\forall g\in\partial f(\bar{x})\}\,,\]
whenever \(\bar{g}\in\ri\partial f(\bar{x})\).
\end{proof}
\begin{cor} Let the assumptions and notation of Theorem \ref{thm:main} hold. Select any $\bar{g} \in \ri\partial h(\Phi(\bar{x}))$. Then
$$\nabla L_{\U}f(\bar{x})=\OU^\top\nabla\Phi(\bar{x})^\top \bar{g}.$$
\end{cor}
\begin{proof}
The proof is immediate from Theorem \ref{thm:main} and Lemma \ref{lem:ulagugrad}. Note that while $\bar{g}$ is not necessarily unique, $\nabla L_{\U}f(\bar{x})$ is \cite[Theorem 4.5]{mifflin2003primal}.
\end{proof}
\noindent We have established the chain rule for PDG fast-track functions. The remainder of this section uses this result to present a separability rule under the same conditions, and Section \ref{sec:examples} provides rules of smooth perturbation and sum of functions.
\begin{lem}[Separability]\label{lem:sep}
Let $f_i:\R^{n_i}\to\overline{\R},$ $f_i=h_i\circ\Phi_i$ satisfy the assumptions of Theorem \ref{thm:main} for each $i\in I=\{1,2,\ldots,k\}$. In particular,
\begin{enumerate}
\item[\rm1.]let $\U_i$ and $\V_i$ be the $\VU$-decomposition of $h_i$ at $\Phi_i(\bar{x}_i)$, and $\OU_i$ be a semiorthonormal basis for $\U_i$;
\item[\rm2.]let $\Phi_i$ be transversal to the manifold $\M_i=\{\bar{x}_i+(u_i+v_i(u_i)):u_i\in\U_i\}$ at $\bar{x}_i$, with $v_i(u_i)$ the fast track of $h_i$ at $\Phi_i(\bar{x}_i)$.
\end{enumerate}
Then the function $f:\R^{n_1}\times\cdots\times\R^{n_k}\to\R$ defined by
$$f(x)=\sum\limits_{i\in I}f_i(x_i)=(h\circ\Phi)(x),$$ where
$h=h_1+\cdots+h_k$ and $\Phi=(\Phi_1,\ldots,\Phi_k),$ satisfies the assumptions of Theorem \ref{thm:chainfasttrack} at $\bar{x}=[\bar{x}_1~~\bar{x}_2~~\cdots~~\bar{x}_k]^\top$, and
$$\nabla_{\U}f(0)=\left(\nabla_{\U_1}f_1(0),\nabla_{\U_2}f_2(0), \ldots,\nabla_{\U_m}f_m(0)\right)^\top,$$
where $\U$ is the diagonal block matrix of $\{\U_i\}_{i\in I}$.
\end{lem}

\begin{proof}
With $u=[u_1~\cdots~u_k]^\top$ and $v=[v_1(u_1)~\cdots~v_k(u_k)]^\top,$ we have that the manifold of $h$ with respect to $\bar{x}$ is $\M=\{\bar{x}+(u+v(u)):u\in\U\}=\M_1\times\cdots\times\M_k$ and $\Phi$ is transversal to $\M$ at $\bar{x}$. Since each $h_i$ is Clarke regular by Theorem \ref{thm:FT=PS} and Definition \ref{df:PS}, we have that
$$\partial h(\bar{x})=\partial h_1(\bar{x}_1)\times\cdots\times\partial h_k(\bar{x}_k),$$
with $h$ Clarke regular by \cite[Proposition 10.5]{rockwets}. Thus, $f$ is partly smooth at $\bar{x}$ and $\M$ defines a fast track for $h$ at $\bar{x}$ by Theorem \ref{thm:FT=PS}(i). Therefore, the assumptions of Theorem \ref{thm:chainfasttrack} are satisfied. With any $\bar{g}_i\in\ri\partial h_i(\Phi_i(\bar{x}_i))$ for each $i\in I$, we have $\bar{g}=[\bar{g}_1~\cdots~\bar{g}_k]^\top\in\ri\partial h(\Phi(\bar{x})).$ Applying Theorem \ref{thm:main}, we have
\begin{align*}
\nabla_{\U}f(0)&=\OU\OU^\top\nabla\Phi(\bar{x})^\top\bar{g}\\
&=\small\left[\begin{array}{c c c c}\OU_1\OU_1^\top&0&\cdots&0\\0&\OU_2\OU_2^\top&\cdots&0\\\vdots&\vdots&\ddots&\vdots\\0&0&\cdots&\OU_k\OU_k^\top\end{array}\right]\left[\begin{array}{c c c c}\nabla\Phi_1(\bar{x}_1)^\top&0&\cdots&0\\0&\nabla\Phi_2(\bar{x}_2)^\top&\cdots&0\\\vdots&\vdots&\ddots&\vdots\\0&0&\cdots&\nabla\Phi_m(\bar{x}_k)^\top\end{array}\right]\left[\begin{array}{c}\bar{g}_1\\\bar{g}_2\\\vdots\\\bar{g}_k\end{array}\right]\\\normalsize
&=\left[\begin{array}{c} \OU_1\OU_1^\top\nabla\Phi_1(\bar{x}_1)^\top\bar{g}_1 \\\OU_2\OU_2^\top\nabla\Phi_2(\bar{x}_2)^\top\bar{g}_2\\
        \vdots \\       \OU_k\OU_k^\top\nabla\Phi_k(\bar{x}_k)^\top\bar{g}_k\end{array}\right]= \left[ \begin{array}{c} \nabla_{\U_1}f_1(0) \\\nabla_{\U_2}f_2(0)\\ \vdots \\ \nabla_{\U_m}f_m(0) \end{array}\right].\qedhere
\end{align*}
\end{proof}

\section{Applications of the chain rule}\label{sec:examples}

In this section, we provide theorems and examples that demonstrate calculus rules (smooth perturbation and sum rules) for the $\U$-gradient. We examine the case of the convex finite-max function, as it is of particular interest in $\VU$-theory.
\begin{thm}[Smooth perturbation]\label{ex:smooth}
Define $f:\R^m\to\overline{\R}$, $f(x)=p(x)+q(x),$ where $p$ is nonsmooth and $q\in\mathcal{C}^2$. Given $\bar{x}\in\dom f,$ let $\OU$ be the $\U$-basis matrix for $p$ at $\bar{x}$. Then $$\nabla_{\U}f(0)=\nabla_{\U}p(\bar{x})+\OU\OU^\top\nabla q(\bar{x})=\nabla_{\U}p(\bar{x})+\OU(\nabla q(\bar{x}))_{\U},$$and the $\U$-space of $f$ at $\bar{x}$ is the $\U$-space of $p$ at $\bar{x}$.
\end{thm}
\begin{proof}
Let
\begin{align*}
\Phi:&\R^m\to\R^{m+1},\Phi(x)=(x,q(x))\\
h:&\R^{m+1}\to\R,h(z_1,z_2)=p(z_1)+z_2,
\end{align*}
so that $f(x)=(h\circ\Phi)(x)$ and $\Phi$ is smooth. Then we have
\begin{align*}
\nabla\Phi(x)&=\left[\begin{array}{c}
\Id_m\\\nabla q(x)^\top\end{array}\right],\\
\partial h(z_1,z_2)&=\left\{\left[\begin{array}{c}
g\\1\end{array}\right]:g\in\partial p(z_1)\right\}\qquad\mbox{\cite[Proposition 10.5]{rockwets}}.
\end{align*}
Denote the manifold of $p$ by $\M_p.$ Then the manifold of $h$ is $\M_h=\{(z_1,z_2):z_1\in\M_p,z_2\in\R\}$. We need to show that $\Phi$ is transversal to $\M_h.$ We have $\ran\nabla\Phi=\R^n\times\alpha\nabla q,\alpha\in\R.$ The tangent cone to $\M_h$ is $T_{\M_h}=T_{\M_p}\times\R,$ so we have
\begin{align*}
\ran\nabla\phi+T_{\M_h}&=(\R^n\times\alpha\nabla q)+(T_{\M_p\times\R}),\\
&=\R^n\times\R.
\end{align*}
Thus, $\Phi$ is transversal to $\M_h$, and the assumptions of Theorem \ref{thm:main} hold. Choose any $\bar{g}\in\ri\partial p(\bar{x}).$ Then by Theorem \ref{thm:main},
\begin{align*}
\nabla_{\U}f(0)&=\OU\OU^\top\nabla\Phi(\bar{x})^\top\left[\begin{array}{c}\bar{g}\\1\end{array}\right],\\
&=\OU\OU^\top\left[\Id_m~~\nabla q(\bar{x})\right]\left[\begin{array}{c}\bar{g}\\1\end{array}\right],\\
&=\OU\OU^\top\bar{g}+\OU\OU^\top\nabla q(\bar{x}),\\
&=\nabla_{\U}p(\bar{x})+\OU\OU^\top\nabla q(\bar{x}).\end{align*}
Then by the definition of $\U$ in Theorem \ref{thm:main}, and denoting the $\U$-space of $f$ and the $\U$-space of $p$ by $\U_f$ and $\U_p$, respectively, we have
\begin{align*}
\U_f&=\left\{d:d^\top\nabla\Phi(\bar{x})^\top \left[\begin{array}{c}g\\1\end{array}\right]=d^\top\Phi(\bar{x})^\top \left[\begin{array}{c}\bar{g}\\1\end{array}\right]~\forall g\in\partial p(\bar{x})\right\},\\
&=\{d:d^\top(g+\nabla q(\bar{x}))=d^\top(\bar{g}+\nabla q(\bar{x}))~\forall g\in\partial p(\bar{x})\},\\
&=\{d:d^\top g=d^\top\bar{g}~\forall g\in\partial p(\bar{x})\}=\U_p,\end{align*}and the $\U$-basis for $p$ at $\bar{x},$ $\OU$, is the $\U$-basis for $f$ at $\bar{x}$ as well.
\end{proof}
A particularly useful application of Theorem \ref{ex:smooth} is the $\ell_2$-regularization of nonsmooth functions. The following corollary states the result.
\begin{cor}
For $p:\R^m\to\overline{\R}$ nonsmooth and $\lambda>0$, define $f$ to be the \emph{$\ell_2$-regularization} of $p$:
$$f=p+\frac{\lambda}{2}\|\cdot\|^2.$$Given $\bar{x}\in\dom p$, let $\OU$ be the $\U$-basis matrix for $p$ at $\bar{x}$. Then
$$\nabla_{\U}f(0)=\nabla_{\U}p(\bar{x})+\lambda\OU\OU^\top\bar{x}=\nabla_{\U}p(\bar{x})+\lambda\OU\bar{x}_{\U},$$and the $\U$-space of $f$ at $\bar{x}$ is the $\U$-space of $p$ at $\bar{x}$.
\end{cor}
\begin{thm}[Sum rule]\label{ex:sum}
For each $i\in I=\{1,2,\ldots,k\},$ let $f_i:\R^n\to\overline{\R}$ be Clarke regular and $\bar{x}\in\dom f_i$. Define$$\Phi:\R^n\to\R^{nk},\Phi(x)=(x,x,\ldots,x)\mbox{ \emph{and} }h_i:\R^n\to\R,h_i(x)=f_i(x).$$ Suppose that each $f_i=h_i\circ\Phi_i$ satisfies the assumptions of Theorem \ref{thm:chainfasttrack} for $\bar{x}$; that is, $\U_i$ and $\V_i$ define the $\VU$-decomposition of $h_i$ at $\bar{x}$, and $\Phi_i$ is transversal to manifold $\M_i$ at $\bar{x}.$
Assume the condition
\begin{equation}\label{eq:sum1}
\sum\limits_{i\in I}z_i=0,z_i\in N_{\M_i}(\bar{x})\mbox{ for each }i\in I\Rightarrow z_i=0\mbox{ for each }i\in I.\end{equation}Define $h:\R^{nk}\to\R,$ $h(z)=\sum_{i\in I}h_i(z_i)$. Then the function $f=h\circ\Phi,$ which simplifies to
$$f(x)=\sum\limits_{i=1}^kf_i(x),$$satisfies the assumptions of Theorem \ref{thm:chainfasttrack}, and
\begin{align}
\U&=\bigcap\limits_{i\in I}\U_i,\label{sumeq1}\\
\nabla_{\U}f(0)&=\OU\OU^\top\sum\limits_{i=1}^k\bar{g}_i=\Proj_{\U}\sum\limits_{i=1}^k\bar{g}_i,\label{sumeq2}\end{align}
for any $\bar{g}_i\in\ri h_i(\Phi_i(\bar{x}))$. 
\end{thm}
\begin{proof}
By a similar argument to the proof of Lemma \ref{lem:sep}, we have that $\M=\M_1\cup\cdots\cup\M_k$ is the manifold of $h$ with respect to $\bar{x}$ and that $\Phi$ is transversal to $\M$ at $\bar{x}$. Each $h_i$ is Clarke regular, so $h=h_1+\cdots+h_k$ is Clarke regular and $\partial h=\partial h_1+\cdots+\partial h_k.$ Thus, $f$ is partly smooth at $\bar{x}$ and $\M$ defines a fast track for $h$ at $\bar{x}$ by Theorem \ref{thm:FT=PS}(i). Therefore, the assumptions of Theorem \ref{thm:chainfasttrack} are satisfied. This, together with \eqref{sumeq1}, gives us the conditions Theorem \ref{thm:main} holds for $f$. We have
\begin{align}
\nabla\Phi(x)&=\left[\Id_n~~\cdots~~\Id_n\right]^\top\in\R^{nk\times n},\label{eq:id1}\\
\partial h(\Phi(x))&=\partial f_1(x)\times\cdots\times\partial f_k(x)\qquad\qquad\mbox{\cite[Proposition 10.5]{rockwets}.}\nonumber\end{align}
For each $i\in I,$ choose any $\bar{g}_i\in\ri\partial f_i(x).$ Then $\bar{g}=\left[\bar{g}_1^\top\cdots\bar{g}_k^\top\right]^\top\in\ri h(\Phi(\bar{x}))$, and
$$\nabla\Phi(\bar{x})^\top \bar{g}=\left[\Id_n~~\cdots~~\Id_n\right]\bar{g}=\sum\limits_{i=1}^k\bar{g}_i.$$Now according to \eqref{eq:Uspace}, we find $\U$ by finding all $d\in\R^n$ such that for all $i\in I$ and for all $g_i\in\partial f_i(\bar{x}),$ denoting $g_1+\cdots+g_k$ by $g\in\partial h(\Phi(\bar{x})),$
$$d^\top\nabla\Phi(\bar{x})^\top g=d^\top\nabla\Phi(\bar{x})^\top \bar{g}.$$By \eqref{eq:id1}, this reduces to
$$\sum\limits_{i=1}^kd^\top(g_i-\bar{g}_i)=0,\qquad\mbox{ for all }g_i\in\partial f_i(\bar{x}).$$
So we have
\begin{equation}\label{uint}
\U=\left\{d:\sum\limits_{i=1}^kd^\top(g_i-\bar{g}_i)=0~\forall g_i\in\partial f_i(\bar{x}),~\forall i\in I\right\}.
\end{equation}
Recall that for each $i\in I,$
$$\U_i=\{d_i\in\R^n:d_i^\top(g_i-\bar{g}_i)=0~\forall g_i\in\partial f_i(\bar{x})\}.$$Let $d\in\bigcap_{i\in I}\U_i.$  Then clearly $d\in\U,$ hence, $\bigcap_{i\in I}\U_i\subseteq\U.$ Let $d\in\U.$ Note that $\bar{g}_i\in\partial f_i(\bar{x})$ for all $i\in I.$ For any arbitrary $j\in I$ fixed, select any $g_j\in\partial f_j(\bar{x})\setminus\{\bar{g}_j\},$ and set $g_i=\bar{g}_i$ for all $i\in I\setminus\{j\}.$ Then the summation of \eqref{uint} reduces to $d^\top(g_j-\bar{g}_j)=0,$ thus, $d\in\U_j.$ Since $j$ is arbitrary in $I$, we have that $d\in\U_i$ for all $i\in I,$ thus, $\U\subseteq\bigcap_{i\in I}\U_i.$ Therefore, \eqref{sumeq1} is true. We have that \eqref{sumeq2} is true by \eqref{eq:Ugrad}.
\end{proof}
\noindent Example \ref{no-cq} shows that \eqref{eq:sum1} is a necessary condition for the sum rule to hold. The inspiration for Theorems \ref{ex:smooth} and \ref{ex:sum} are Corollaries 4.6 and 4.7 of \cite{Lewis02}, where the author develops smooth perturbation and sum rules for partly smooth functions.  This allows us to determine the $\U$-gradients and $\U$-spaces for smooth perturbations and sums of PDG functions.\par  Example \ref{ex:ell1reg}, on the $\ell_1$-regularization problem, provided
a PDG structure that is not strongly transversal. As mentioned, this does not preclude
the $\ell_1$-norm from having a fast track at every point. Using the sum rule, we can now show that the $\ell_1$-regularization problem has a fast track and determine the corresponding $\U$-gradient. 
\begin{ex}\label{ex:ell1fasttrack} Let $f:\R^n\to\overline{\R}$ be $\mathcal{C}^2$ and $\tau>0$. Then the function $f+\tau\|\cdot\|_1$ has a fast track at any point $\bar{x}\in\dom f$.  
\end{ex}
\begin{proof}Set $h_1=f$, $\Phi_1=\Id$, $h_2=\tau\|\cdot\|_1$ and $\Phi_2=\Id$. We will use the sum rule on these two component functions.\\
\noindent Note that $h_1 \in \mathcal{C}^2$ implies $\U = \R^n$ and $\V=0$.  The $\U$-Lagrangian is given by
$$L_{\U}h_1(u;\bar{g})=\min_{v \in \V} \{f_1(\bar{x} + (\OU u+\OV v)) -\bar{g}^\top\OV v\}=h_1(\bar{x} + u),$$
which is $\mathcal{C}^2$ in $\U$.  The solution mapping to the $\U$-Lagrangian is trivially
$$W_{\U}h_1(u;\bar{g})=\arg\min_{v \in \V} \{h_1(\bar{x} + (\OU u+\OV v)) -\bar{g}^\top\OV v\}=\{0\}.$$
Hence, $v(u) = 0$ is a $\mathcal{C}^2$ selection of $W_{\U}h_1(u;\bar{g})$, and therefore $v(u)$ is a fast track of $h_1$ at $\Phi_1(\bar{x})$.  Since $\ran(\nabla\Phi_1(\bar{x}))=\R^n$, transversality holds.  Thus, $f_1 = h_1 \circ \Phi_1$ satisfies the conditions of Theorem \ref{thm:chainfasttrack} with $\M_1 = \R^n$.\\
\noindent Turning our attention to $f_2 = h_2 \circ \Phi_2$, notice that
$$\partial (h_2)(\bar{x})=\tau\left\{g:\begin{array}{rl}g_i=\sgn(\bar{x}_i)&\mbox{if }\bar{x}_i\neq0,\\g_i=[-1, 1]&\mbox{if }\bar{x}_i=0.\end{array}\right\}.$$
This gives
$$\V=\{v\in\R^n :v_i=0\mbox{ whenever }\bar{x}_i\neq0\}\quad\mbox{ and }\quad\U= \{u\in\R^n:u_i=0~\mbox{whenever}~\bar{x}_i = 0 \}.$$
Let $\OU$ and $\OV$ be the corresponding basis matrices. Define $\bar{g}$ such that 	
$\bar{g}_i=\tau\sgn(\bar{x}_i)$ and notice that $\bar{g} \in \ri \partial (\|\cdot\|_1)(\bar{x})$.   
Also, notice that given any $v \in \V$, we have that $\bar{g}^\top \OV v = 0$.
As such, the solution mapping of the $\U$-Lagrangian is given by
$$\begin{array}{rcl}W_{\U}h_2(u;\bar{g})&=&\min_{v\in\V}\{\|\bar{x}+(\OU u+\OV v)\|_1-\bar{g}^\top\OV v\},\\&=&\arg\min_{v\in\V}\{\|\bar{x}+(\OU u+\OV v)\|_1\}=\{0\}.\end{array}$$
Hence, $v(u)$ is a $\mathcal{C}^2$ selection of $W_{\U}h_2(u;\bar{g})$.  Applying this to the $\U$-Lagrangian, we find that 
$$\begin{array}{rcl}L_{\U}f(u;\bar{g})&=&\min_{v\in\V}\{\|\bar{x}+(\OU u+\OV v)\|_1-\bar{g}^\top\OV v\},\\&=&\|\bar{x}+\OU u\|_1=\sum\limits_{i\notin A(\bar{x})}|\bar{x}_i+u_i|,\end{array}$$where $A(\bar{x})=\{i:\bar{x}_i=0\}$ is the active set at $\bar{x}$. Notice that $L_{\U}f(\,\cdot~;\bar{g})$ is (locally) $\mathcal{C}^2$, since $i \notin A(\bar{x})$ implies that $|\bar{x}_i + u_i|$ is linear near $0$.  Finally, since $\ran(\nabla\Phi_1(\bar{x}))=\R^n$, transversality holds, so $f_2 = h_2 \circ \Phi_2$ satisfies the conditions of Theorem \ref{thm:chainfasttrack}.

Since $\M_1 = \R^n$, we have that \eqref{eq:sum1} holds and we may now apply the sum rule.  We find that $f(x)+\tau\|x\|_1$ satisfies the conditions of Theorem \ref{thm:chainfasttrack}, so it has a fast track and we have
\begin{align*}\U=\U_1\cap\U_2=\U_2&=\{ u\in \R^n : u_i = 0~\mbox{whenever}~\bar{x}_i = 0 \},\\
\nabla_{\U}(f+\tau\|\cdot\|_1)(0)&=\OU\OU^\top\sum\limits_{i=1}^k\bar{g}_i=\Proj_{\U}\nabla f(\bar{x}).\end{align*}
\end{proof}
A prominent function that has the format of Example \ref{ex:ell1fasttrack} is the LASSO function. By applying the example to $f(x)=\frac{1}{2}\|Ax-b\|_2^2$, we obtain the following result.
\begin{cor}[LASSO]
Consider the LASSO problem
$$\min\left\{\frac{1}{2}\|Ax-b\|^2_2+\tau\|x\|_1\right\}.$$This function has a fast track at any point $\bar{x}\in\R^n$, and$$\nabla_{\U}\left(\frac{1}{2}\|A\cdot-b\|_2^2+\tau\|\cdot\|_1\right)(0)=\Proj_{\U}(A^\top A\bar{x}-A^\top b),$$
$$\mbox{\emph{where} }\U=\{u\in\R^n:u_i=0\mbox{ \emph{whenever} }\bar{x}_i=0\}.$$
\end{cor}
The following example illustrates the fact that \eqref{eq:sum1} is necessary for the sum rule to function properly. We construct functions that do not comply with \eqref{eq:sum1}, and show that the sum rule fails.
\begin{ex}\label{no-cq}
Following the notation of Theorem \ref{ex:sum}, let $f_1,f_2:\R^2\to\overline{\R}$ be defined by
$$f_1(x,y)=\iota_B=\begin{cases}0,&\mbox{if }(x,y)\in B,\\\infty,&\mbox{if }(x,y)\not\in B,\end{cases}\qquad f_2(x,y)=\iota_L=\begin{cases}0,&\mbox{if }(x,y)\in L,\\\infty,&\mbox{if }(x,y)\not\in L,\end{cases},$$where $B=\{(x,y):\|(x,y)-(1,0)\|=1\}$ and $L=\{(0,y):y\in\R\}$. Then the sum rule does not apply at $(\bar{x},\bar{y})=(0,0)$.
\end{ex}
\begin{proof}
At $(\bar{x},\bar{y})$, we find that the manifolds of $f_1$ and $f_2$, respectively, are
\begin{equation*}
\M_1=B,\qquad\M_2=L.
\end{equation*}
At the origin, we have
$$\U_1=\{(0,y):y,z\in\R\},~\U_2=\{(0,y):y\in\R\}\Rightarrow\U_1\cap\U_2=\U_2.$$
Hence, the normal cones to the manifolds at the origin are
\begin{equation*}
N_{\M_1}(\bar{x},\bar{y})=\{(x,0):x\in\R\},\qquad N_{\M_2}(\bar{x},\bar{y})=\{(x,0):x\in\R\}.
\end{equation*}
Then, we observe that
$$(1,0)\in N_{\M_1}(\bar{x},\bar{y}),~(-1,0)\in N_{\M_2}(\bar{x},\bar{y}),\mbox{ and }(1,0)+(-1,0)=(0,0).$$
Thus, $f_1$ and $f_2$ do not comply with \eqref{eq:sum1} at $(\bar{x},\bar{y})$. 
However, $$f(x,y)=f_1(x,y)+f_2(x,z)=\iota_{B\cap L}=\iota_{(0,0)},$$ so $\U=\{0\}$ at $(\bar{x},\bar{y})$. The sum rule fails.
\end{proof}
We finish this section with an example of a useful family of functions, the convex finite-max functions. These functions are commonly used in $\VU$-theory, because they are simple enough to manage and complex enough to showcase the intricacies of the theory. We verify that the $\U$-gradient and $\U$-space are what we expect them to be in this special case.
\begin{ex}[Convex finite-max]\label{ex:finmax}
Let $f=h\circ\Phi$ be a convex finite-max function on $\R^n$:
$$\Phi(x)=(\Phi_1(x), \Phi_2(x), \ldots, \Phi_m(x))\in\mathcal{C}^2,\qquad h(y)=\max y_i,$$with $\Phi_i$ convex for each $i$. For $\bar{x}\in\dom f,$ suppose that the active set $A(\bar{x})=\{i:\Phi_i(\bar{x})=\Phi(\bar{x})\}$ defines an affinely independent set of subfunction gradients $\nabla\Phi_i$. Then the assumptions of Theorem \ref{thm:main} hold and
\begin{align}
\U&=\Span\limits_{i\in A(\bar{x})}\left[\begin{array}{c}\nabla\Phi_i(\bar{x})-\nabla\Phi_j(\bar{x})\end{array}\right]\mbox{ \emph{with} }j\in A(\bar{x})\mbox{ \emph{fixed},}\label{cfm1}\\
\nabla_{\U}f(0)&=\OU\OU^\top\nabla\Phi(\bar{x})^\top \sum\limits_{i\in A(\bar{x})}\alpha_ie_i,\label{cfm2}
\end{align}
where $e_i\in\R^n$ is the $i$\textsuperscript{th} canonical vector and
$$\alpha_i\geq0~\forall i\in A(\bar{x}),~~\sum\limits_{i\in A(\bar{x})}\alpha_i=1.$$
\end{ex}
\begin{proof}
By Definitions \ref{df:pdg}, \ref{strongtrans} and Theorem \ref{thm:fast track}, $f$ is a PDG function with a fast track at $\bar{x}$ relative to the manifold $\M$. Note that $\M=\R^n$ if $|A(\bar{x})|=1$, and if $|A(\bar{x})|>1$, then $$\M=\{x\in\R^n:\Phi_i(x)=\Phi_j(x),i,j\in A(\bar{x}),i\neq j\}.$$ Since $h$ is a maximum of linear functions (i.e. a polyhedral function), $h$ is partly smooth at $\Phi(\bar{x})$ relative to $\M$ \cite[Example 3.4]{Lewis02} and we have
\begin{equation}\label{eq:parh}
\partial h(y)=\conv\limits_{i\in A(y)}\nabla y_i=\conv\limits_{i\in A(y)}e_i.\end{equation}
Since the set $\{\nabla\Phi_i(\bar{x}):i\in A(\bar{x})\}$ is linearly independent, $\Phi$ is transversal to $\M$ at $\bar{x}$. Thus, the assumptions of Theorem \ref{thm:main} hold. By \eqref{eq:Uspace}, we have
$$\U = \{d \in \R^m : d^\top\nabla \Phi(\bar{x})^\top g  = d^\top\nabla\Phi(\bar{x})^\top  \bar{g}~\mbox{ for all }~g \in \partial h(\Phi(\bar{x})) \},$$which proves \eqref{cfm1}. Then $\bar{g}=\sum_{i\in A(\bar{x})}\alpha_ie_i\in\ri h(\Phi(\bar{x}))$. Applying Theorem \ref{thm:main} completes the proof of \eqref{cfm2}.
\end{proof}

\section{Conclusion}\label{sec:conclusion}

We have established calculus rules for the $\VU$-decomposition of primal-dual gradient structured functions with fast tracks. In doing so, we introduced some new notation to make the inner workings of $\VU$-theory clearer. We refer to the objects $x_{\V}$ and $x_{\U}$ as \emph{restrictions} of vector $x$ to the $\V$-space and the $\U$-space, respectively, rather than use the term \emph{projections} that is found in existing literature. This allows for the term \emph{projection} to be used in the traditional sense as the orthogonal projection of a vector onto a set. We also introduced the $\U$-gradient as a separate object from the gradient of the $\U$-Lagrangian and presented the relationship between the two. The difference between transversality and nondegeneracy was identified and equivalence conditions were given.

One avenue of further investigation is the second-order calculus of $\VU$-theory. The first-order rules are provided here, but properties and characterizations of the $\U$-Hessian would also be useful. At the moment, it is unclear in exactly which direction to head when forming a chain rule for the $\U$-Hessian, since there are many definitions of second-order differentiability \cite[\S 13]{rockwets}. We leave the matter for future consideration.

\section*{Acknowledgement} The authors wish to thank the anonymous referee for their diligent work on this paper and in particular for their help with improving the presentation of Lemma \ref{lem:horizon}.

\small
\bibliographystyle{alpha}

\begin{thebibliography}{}

\end{thebibliography}


\begin{thebibliography}{AYRP15}

\bibitem[AYRP15]{akbari2015new}
Z.~Akbari, R.~Yousefpour, and M.~Reza~Peyghami.
\newblock A new nonsmooth trust region algorithm for locally {L}ipschitz
  unconstrained optimization problems.
\newblock {\em J. Optim. Theory Appl.}, 164(3):733--754, 2015.

\bibitem[BCI11]{bello2011strongly}
J.~Bello~Cruz and A.~Iusem.
\newblock A strongly convergent method for nonsmooth convex minimization in
  {H}ilbert spaces.
\newblock {\em Numer. Funct. Anal. Optim.}, 32(10):1009--1018, 2011.

\bibitem[BLO05]{burke2005robust}
J.~Burke, A.~Lewis, and M.~Overton.
\newblock A robust gradient sampling algorithm for nonsmooth, nonconvex
  optimization.
\newblock {\em SIAM J. Optim.}, 15(3):751--779, 2005.

\bibitem[BM88]{burke1988identification}
J.~Burke and J.~Mor{\'e}.
\newblock On the identification of active constraints.
\newblock {\em SIAM J. Num. Anal.}, 25(5):1197--1211, 1988.

\bibitem[BPC{\etalchar{+}}11]{Boydetal2011}
S.~Boyd, N.~Parikh, E.~Chu, B.~Peleato, and J.~Eckstein.
\newblock Distributed optimization and statistical learning via the alternating
  direction method of multipliers.
\newblock {\em Found. Trends Mach. Learn.}, 3(1):1--122, January 2011.

\bibitem[BS09]{belloni2009dynamic}
A.~Belloni and C.~Sagastiz\'abal.
\newblock Dynamic bundle methods.
\newblock {\em Math. Program.}, 120(2, Ser. A):289--311, 2009.

\bibitem[CGT00]{conn2000trust}
A.~Conn, N.~Gould, and P.~Toint.
\newblock {\em Trust-region methods}.
\newblock MPS/SIAM Series on Optimization. Society for Industrial and Applied
  Mathematics (SIAM), Philadelphia, PA; Mathematical Programming Society (MPS),
  Philadelphia, PA, 2000.

\bibitem[Che12]{chen2012smoothing}
X.~Chen.
\newblock Smoothing methods for nonsmooth, nonconvex minimization.
\newblock {\em Math. Program.}, 134(1, Ser. B):71--99, 2012.

\bibitem[CM87]{calamai1987projected}
P.~Calamai and J.~Mor{\'e}.
\newblock Projected gradient methods for linearly constrained problems.
\newblock {\em Math. Program.}, 39(1):93--116, 1987.

\bibitem[dSYS97]{desampaio1997trust}
R.~de~Sampaio, J.~Yuan, and W.~Sun.
\newblock Trust region algorithm for nonsmooth optimization.
\newblock {\em Appl. Math. Comput.}, 85(2-3):109--116, 1997.

\bibitem[Dun87]{dunn1987convergence}
J.~Dunn.
\newblock On the convergence of projected gradient processes to singular
  critical points.
\newblock {\em J. Optim. Theory Appl.}, 55(2):203--216, 1987.

\bibitem[EB92]{eckstein1992douglas}
J.~Eckstein and D.~Bertsekas.
\newblock On the {D}ouglas--{R}achford splitting method and the proximal point
  algorithm for maximal monotone operators.
\newblock {\em Math. Program.}, 55(1-3):293--318, 1992.

\bibitem[Fl{\aa}92]{flaam1992finite}
S.~Fl{\aa}m.
\newblock On finite convergence and constraint identification of subgradient
  projection methods.
\newblock {\em Math. Program.}, 57(1):427--437, 1992.

\bibitem[Har06]{Hare06}
W.~Hare.
\newblock Functions and sets of smooth substructure: relationships and
  examples.
\newblock {\em Comput. Optim. Appl.}, 33(2-3):249--270, 2006.

\bibitem[Har14]{Hare2014}
W.~Hare.
\newblock Numerical analysis of $\mathcal{V}\mathcal{U}$-decomposition,
  $\mathcal{U}$-gradient, and $\mathcal{U}$-{H}essian approximations.
\newblock {\em SIAM J. Optim.}, 24(4):1890--1913, 2014.

\bibitem[HL07]{hare2007identifying}
W.~Hare and A.~Lewis.
\newblock Identifying active manifolds.
\newblock {\em Alg. Oper. Res.}, 2(2):75, 2007.

\bibitem[HP18]{hare2018computing}
W.~Hare and C.~Planiden.
\newblock Computing proximal points of convex functions with inexact subgradients.
\newblock {\em Set-valued Var. Anal.}, 26(3):469--492, 2018.

\bibitem[HPS19]{hare2019derivative}
W.~Hare and C.~Planiden and C.~Sagastiz{\'a}bal.
\newblock A derivative-free $\VU$-algorithm for convex finite-max functions.
\newblock {\em arXiv preprint arXiv:1903.11184.}, 2019.

\bibitem[HS08]{hou2008global}
L.~Hou and W.~Sun.
\newblock On the global convergence of a nonmonotone proximal bundle method for
  convex nonsmooth minimization.
\newblock {\em Optim. Methods Softw.}, 23(2):227--235, 2008.

\bibitem[HS10]{proxbund}
W.~Hare and C.~Sagastiz{\'a}bal.
\newblock A redistributed proximal bundle method for nonconvex optimization.
\newblock {\em SIAM J. Optim.}, 20(5):2442--2473, 2010.

\bibitem[HSS16]{hareetal16}
W.~Hare, C.~Sagastiz{\'a}bal, and M.~Solodov.
\newblock A proximal bundle method for nonsmooth nonconvex functions with
  inexact information.
\newblock {\em Comput. Optim. Appl.}, 63(1):1--28, 2016.

\bibitem[KBM12]{karmitsa2012comparing}
N.~Karmitsa, A.~Bagirov, and M.~M\"akel\"a.
\newblock Comparing different nonsmooth minimization methods and software.
\newblock {\em Optim. Methods Softw.}, 27(1):131--153, 2012.

\bibitem[Kiw85]{kiwiel1985methods}
K.~Kiwiel.
\newblock {\em Methods of descent for nondifferentiable optimization}, volume
  1133 of {\em Lecture Notes in Mathematics}.
\newblock Springer-Verlag, Berlin, 1985.

\bibitem[Lew02]{Lewis02}
A.~Lewis.
\newblock Active sets, nonsmoothness, and sensitivity.
\newblock {\em SIAM J. Optim.}, 13(3):702--725, 2002.

\bibitem[LOS00]{ulagconv}
C.~Lemar{\'e}chal, F.~Oustry, and C.~Sagastiz{\'a}bal.
\newblock The {$U$}-{L}agrangian of a convex function.
\newblock {\em Trans. Amer. Math. Soc.}, 352(2):711--729, 2000.

\bibitem[LS18]{liu-sagastizabal}
Sh. Liu and C.~Sagastiz{\'a}bal.
\newblock Beyond first order: {$\VU$}-decomposition methods.
\newblock Chapter in Springer book ``Special methods for nonsmooth
  optimization", 2018.
  
\bibitem[LS97]{lemarechal-sagastizabal-1997}
C.~Lemar{\'e}chal and C.~Sagastiz{\'a}bal.
\newblock Practical aspects of the {Moreau-Yosida} regularization: theoretical
  preliminaries.
\newblock {\em SIAM J. Optim.}, 7(2):367--385, 1997.

\bibitem[MJ11]{min2011linearly}
S.~Min and L.~Jing.
\newblock A linearly convergent conjugate gradient method for unconstrained
  optimization problems.
\newblock {\em Acta Math. Vietnam.}, 36(3):611--622, 2011.

\bibitem[MM97]{martinez1997trust}
J.~Mart\'inez and A.~Moretti.
\newblock A trust region method for minimization of nonsmooth functions with
  linear constraints.
\newblock {\em Math. Program.}, 76(3, Ser. B):431--449, 1997.

\bibitem[MS00a]{mifflin2000functions}
R.~Mifflin and C.~Sagastiz\'abal.
\newblock Functions with primal-dual gradient structure and $\U$-Hessians.
\newblock In {\em Nonlinear optimization and related topics ({E}rice, 1998)},
  volume~36 of {\em Appl. Optim.}, pages 219--233. Kluwer Acad. Publ.,
  Dordrecht, 2000.

\bibitem[MS00b]{vutheory}
R.~Mifflin and C.~Sagastiz\'abal.
\newblock On $\VU$-theory for functions with primal-dual gradient structure.
\newblock {\em SIAM J. Optim.}, 11(2):547--571 (electronic), 2000.

\bibitem[MS02]{mifflin2002proximal}
R.~Mifflin and C.~Sagastiz\'abal.
\newblock Proximal points are on the fast track.
\newblock {\em J. Convex Anal.}, 9(2):563--579, 2002.
\newblock Special issue on optimization (Montpellier, 2000).

\bibitem[MS03]{mifflin2003primal}
R.~Mifflin and C.~Sagastiz\'abal.
\newblock Primal-dual gradient structured functions: second-order results;
  links to epi-derivatives and partly smooth functions.
\newblock {\em SIAM J. Optim.}, 13(4):1174--1194, 2003.

\bibitem[MS04]{vusmoothness}
R.~Mifflin and C.~Sagastiz\'abal.
\newblock $\VU$-smoothness and proximal point results for some nonconvex
  functions.
\newblock {\em Optim. Methods Softw.}, 19(5):463--478, 2004.

\bibitem[MS05]{vualg}
R.~Mifflin and C.~Sagastiz\'abal.
\newblock A $\VU$-algorithm for convex minimization.
\newblock {\em Math. Program.}, 104(2-3, Ser. B):583--608, 2005.

\bibitem[MS99]{vudecomp}
R.~Mifflin and C.~Sagastiz\'abal.
\newblock $\VU$-decomposition derivatives for convex max-functions.
\newblock In {\em Ill-posed variational problems and regularization techniques
  ({T}rier, 1998)}, volume 477 of {\em Lecture Notes in Econom. and Math.
  Systems}, pages 167--186. Springer, Berlin, 1999.

\bibitem[Nes05]{nesterov2005smooth}
Y.~Nesterov.
\newblock Smooth minimization of non-smooth functions.
\newblock {\em Math. Program.}, 103(1, Ser. A):127--152, 2005.

\bibitem[Ove92]{overton1992large}
M.~Overton.
\newblock Large-scale optimization of eigenvalues.
\newblock {\em SIAM J. Optim.}, 2(1):88--120, 1992.

\bibitem[RW98]{rockwets}
R.~Rockafellar and R.~Wets.
\newblock {\em Variational Analysis}.
\newblock Grundlehren der Mathematischen Wissenschaften. Springer-Verlag,
  Berlin, 1998.

\bibitem[Sag18]{sagastizabal-icm2018}
C. Sagastiz\'abal.
\newblock {A $\VU$ point of view of nonsmooth optimization}.
\newblock In {\em Proceedings of the International Congress of Mathematicians
  2018- Invited Lectures}, volume~3, pages 3785--3806, 2018.

\bibitem[Sha03]{shapiro2003class}
A.~Shapiro.
\newblock On a class of nonsmooth composite functions.
\newblock {\em Math. Oper. Res.}, 28(4):677--692, 2003.

\bibitem[Teb18]{Teboulle2018}
M. Teboulle.
\newblock A simplified view of first order methods for optimization.
\newblock {\em Math. Program.}, 170(1):67--96, Jul 2018.

\bibitem[WCP17]{wen2017linear}
B.~Wen, X.~Chen, and T.~Pong.
\newblock Linear convergence of proximal gradient algorithm with extrapolation
  for a class of nonconvex nonsmooth minimization problems.
\newblock {\em SIAM J. Optim.}, 27(1):124--145, 2017.

\bibitem[Wil65]{wilkinson1965algebraic}
J.~Wilkinson.
\newblock {\em The algebraic eigenvalue problem}, volume~87.
\newblock Clarendon Press Oxford, 1965.

\bibitem[XNY15]{xiao2015convergence}
J.~Xiao, M.~Ng, and Y.~Yang.
\newblock On the convergence of nonconvex minimization methods for image
  recovery.
\newblock {\em IEEE Trans. Image Process.}, 24(5):1587--1598, 2015.

\end{thebibliography}

\newcommand{\etalchar}[1]{$^{#1}$}
\def\cprime{$'$}

\end{document}